\documentclass[a4paper,10pt,twoside]{article}
\RequirePackage{amsthm,amsmath,amsfonts,amssymb}
\RequirePackage[numbers,sort&compress]{natbib}
\RequirePackage{hyperref}
\RequirePackage{graphicx}
\usepackage[a4paper,portrait,left=3.5cm,right=3.5cm,top=3.5cm,bottom=3.5cm]{geometry}
\usepackage[utf8]{inputenc}
\usepackage{fancyhdr}
\usepackage{mathrsfs,dsfont}
\usepackage{stmaryrd} 
\usepackage{color}
\usepackage{pgfplots}
\pgfplotsset{compat=1.16}
\usetikzlibrary{arrows}

\definecolor{dgreen}{cmyk}{0.37,0,0.39,0.61}
\definecolor{dblue}{cmyk}{0.44,0.44,0,0.55}
\hypersetup{colorlinks=true,citecolor=dgreen,linkcolor=dblue,urlcolor=dblue, 
  pdftitle={The fluctuation process of the giant component}, 
  pdfauthor={N. Enriquez, G. Faraud and S. Lemaire} }
  
\newtheorem{thm}{Theorem}[section]
\newtheorem{prop}[thm]{Proposition}
\newtheorem{lem}[thm]{Lemma}
\newtheorem*{tref}{Theorem}
\theoremstyle{remark}
\newtheorem{rem}[thm]{Remark}

\newcommand\var{\mathrm{Var}}
\newcommand\NN{\mathbb{N}}
\newcommand\RR{\mathbb{R}}
\newcommand\mc[1]{\mathcal{#1}} 
\newcommand\un{\mathds{1}}
\newcommand\ER{\mathrm{ER}}
\newcommand\Bin{\mathrm{Bin}}

\newcommand\BGW{\mathrm{BGW}}
\newcommand\cG[1][n]{\mathscr{G}_{#1}} 
\newcommand\cGt[1][n]{\tilde{\mathscr{G}}_{#1}} 
\newcommand\Lt{L_n} 
\newcommand\Deltat{\tilde{\Delta}L_n} 
\newcommand\Pt{\mathbb{P}_{t}}
\newcommand\Et{\mathbb{E}_{t}}
\newcommand\Pv[1][t]{\mathbb{P}_{v^{(-1)}(#1)}}
\newcommand\Ev[1][t]{\mathbb{E}_{v^{(-1)}(#1)}}
\renewcommand\P{\mathbb{P}}
\renewcommand\epsilon{\varepsilon}
\newcommand\E{\mathbb{E}}
\newcommand\Xmu{(uX_n)}
\newcommand\cC[1]{\mathscr{C}_{n,#1}} 
\newcommand\cm{\mathscr{C}^{max}_{n}} 
\newcommand{\enu}[1][n]{\llbracket {#1} \rrbracket}

\begin{document}
\title{The process of fluctuations of the giant component of an Erd\H{o}s-R\'enyi graph}

\author{Nathana\"{e}l Enriquez\thanks{Université Paris-Saclay, CNRS, Laboratoire de mathématiques d'Orsay, 91405, Orsay, France;
DMA, \'Ecole normale sup\'erieure, Universit\'e PSL, CNRS, 75005 Paris, France; email: \url{nathanael.enriquez@universite-paris-saclay.fr}}, Gabriel Faraud\thanks{Université Paris Nanterre, CNRS, Laboratoire MODAL'X, 92000, Nanterre, France;
email: \url{gabriel.faraud@parisnanterre.fr}}, Sophie Lemaire\thanks{Université Paris-Saclay, CNRS, Laboratoire de mathématiques d'Orsay, 91405, Orsay, France;  email: \url{sophie.lemaire@universite-paris-saclay.fr}.}}
\date{}

\maketitle
\begin{abstract} We present a detailed study of the evolution of the giant component of the Erd\H{o}s-R\'enyi graph process as the mean degree increases from 1 to infinity. It leads to the identification of the limiting process of the rescaled fluctuations of its order around its deterministic asymptotic. This process is Gaussian with an explicit covariance.
\end{abstract}

\bigskip
\maketitle

\section{Introduction}

The Erd\H{o}s-R\'enyi graph $\ER(n,p_n)$  is, without any contest, the most popular random graph. The model goes back to Gilbert in \cite{Gilbert} and was first studied by Erd\H{o}s and R\'enyi, who presented the main features of a slight variation of this model in a seminal paper \cite{ErdosRenyi}. For an integer $n\geq1$ and $p_n\in (0,1)$, it is defined as a graph with $n$ vertices, where any pair of vertices are connected by an edge independently with probability $p_n$. The most striking phenomenon occurs at the scale $p_n=\frac{c}{n}$ where this model exhibits a  phase transition at $c=1$. 
Namely, with high probability,  for $c<1$, all connected components of the graph are of order $O(\ln(n))$, whereas for $c>1$, a single component is of order of magnitude $n$ and all the others are of order $O(\ln(n))$. More precisely, the largest component, commonly called the `{\it giant component}' is asymptotically of order $\rho(c) n$ where $\rho(c)$ denotes the probability that a Galton-Watson tree with Poisson reproduction distribution of parameter $c$ is infinite. The constant $\rho(c)$ is the only solution in $(0,1)$ of the equation 
\begin{equation}
\label{eqsurviv}
1-x=e^{-cx}.
\end{equation}

Few years later, the question of the fluctuations of this giant component in the case $c>1$ was raised. Different approaches have been proposed to prove that the fluctuation of the order (that is, the number of vertices) of the giant component around $\rho(c)n$ is asymptotically Gaussian with variance $\dfrac{\rho(c)(1-\rho(c))}{(1-c(1-\rho(c)))^2}n$.  The first proof was written by Stepanov \cite{Stepanov}, who adopted a combinatorial approach and sketched the proof of a local limit theorem. Later, Pittel \cite{Pittel} obtained the asymptotic normality of the giant component from the asymptotic normality of the number of components of any given tree-like shape. 
Barraez, Boucheron and de la Vega \cite{BBDLV} further established an upper bound for the rate of convergence and a bivariate result for both the order of the giant component and the number of edges by analysing an exploration process of the graph. Independently, Puhalskii in \cite{Puhalskii} also used an exploration process to study the joint distribution of three variables: the giant component, the number of excess edges and the number of connected components. The same year, Pittel and Wormald in \cite{PittelWormald} extended the study of the fluctuations of the giant component for $c=1+\epsilon_n$ with $\epsilon_n\rightarrow 0$ and $\epsilon_n n^{1/3}\rightarrow +\infty$. Using combinatorial methods, they showed a trivariate result on the fluctuation of the 2-core that implies that the asymptotic fluctuation of the giant component is Gaussian with variance $\frac{2n}{\epsilon_n}$. Later, this last result was directly proven by Bollobas and Riordan in \cite{BollobasRiordan} with the help of an exploration process. 

All the previously mentioned articles focus on the study of the Erd\H{o}s-R\'enyi graph for a {\it fixed parameter} $c$. There is however a very natural construction that makes it possible to couple the graphs $\ER(n,\frac cn)$ for different values of $c$ and  obtain a graph-valued process. Namely, we can attach to every pair $(i,j)$ of vertices a weight given by independent uniform random variables $w_{i,j}$. We consider the graph $\cG(t)$ whose edges are the pairs of vertices with weight smaller than $\frac{t}{n}$. With this construction, for any $t$, $\cG(t)$ follows the distribution of an $\ER(n,\frac tn)$ graph. There are remarkably few results concerning the dynamics of the features of the graph $\cG(t)$, seen as a stochastic process. An important reference on these questions are the works by Janson \cite{Janson90,JansonMAMS,JansonMST}. In \cite{Janson90,JansonMAMS}, Janson develops a sophisticated general method based on chaos decomposition of martingales to get functional limit theorems for random graph statistics. In \cite{JansonMST}, he uses this method to establish that the infinite family of processes, defined by the number of isolated trees of any given order, has Gaussian fluctuations with explicit asymptotic covariances. He conjectures (\cite{JansonMST}, Remark 1.3) that this approach could be used to show that the fluctuations of the order of the giant component converge to a Gaussian process. We are going to prove this result and identify the limiting process.  Our approach however does not rely on the techniques proposed by Janson; it relies on a careful analysis of the infinitesimal increment of the giant component and  leads to a stochastic differential equation that we can explicitly solve.

In order to state the precise result, let us introduce some notation. 
Throughout the article we let $\cC{x}(t)$ define the connected component of the vertex $x$ in $\cG(t)$ and  $|\cC{x}(t)|$ its order, that is its number of vertices. We set $L_n(t)=\max(|\cC{x}(t)|,\, x\in\enu)$. Among components of order $L_n(t)$, we  choose one with a fixed rule (for instance the component of the first vertex having  $L_n(t)$  vertices) and denote it $\cm(t)$.  

We deal with the fluctuation process of $L_n(t)$, namely  $X_n(t):=\frac{1}{\sqrt{n}}(L_n(t)-n\rho(t))$ where the function $\rho$ is defined by \eqref{eqsurviv}. As proven in the above references, it converges in distribution for a {\it fixed} $t>1$, to a Gaussian distribution with variance 
\begin{equation}
\label{variance}\sigma^{2}(t)=\frac{\rho(t)(1-\rho(t))}{(1-t(1-\rho(t)))^2}. 
\end{equation}
 
Our main result identifies the limit of  $X_n$ as a {\it process}. 

\begin{thm}
\label{mainthm}
Let $\{B(t),\,t\geq 0\}$ be a standard Brownian motion on $\RR$.
Set for all $t>1$,
\[u(t)=\frac{1}{1-\rho(t)}-t \text{ and } v(t)=\frac{1}{1-\rho(t)}-1. \]
Let $t_0,t_1$ be two real numbers such that $1<t_0<t_1$.\\
The process $\{X_n(t),\; t\in[t_0,t_1]\}$ converges weakly, with respect to the Skorohod topology, to the Gaussian process $\{X(t), \,t\in[t_0,t_1]\}$ defined by:
\begin{equation}
\label{eq:brownian}
\forall t\in [t_0,t_1], \quad X(t)=\frac{1}{u(t)}B(v(t)). 
\end{equation} 
\end{thm}

Note that the limiting process is Gaussian and does not have independent increments. Its variance at time $t$ is $\sigma^2(t)$; it explodes at $1^+$ and is bounded away from 0 and $+\infty$ on the interval $[t_0,t_1]$. 

A heuristic of how the giant component evolves during an infinitesimal time interval is displayed in Section \ref{heuristique} in order to identify the limiting process $X$. The first tool to prove Theorem  \ref{mainthm} is a representation of the variation of the order of the giant component during a small time interval detailed in Section \ref{representation}.  We show in Section \ref{convergencefdd} the finite-dimensional convergence of the sequence $((uX_n)\circ v^{(-1)})_n$ towards a standard Brownian motion, based on an extension of the three-moments method of Norman \cite{norman}.
Finally, we prove its tightness in Section \ref{Tightness}, which finishes the proof of Theorem \ref{mainthm}.

\section{Heuristic} \label{heuristique}

We propose a heuristic to identify the limiting process of $X_n$ as a solution of a stochastic differential equation (SDE). Our main idea is to understand precisely what happens to the largest component during an infinitesimal step $h$. Between time $t$ and $t+h$ a random number $N_n(t)$ of edges is added to the graph. This number is distributed as a $\Bin(\frac12n(n-1),\frac{h}{n})$ random variable. 

Each edge contributes to increasing the order of the giant component between $t$ and $t+h$ if it connects one vertex of the largest component to a vertex outside of it. This event has probability
\[
\zeta_n(t):=\frac{2L_n(t)(n-L_n(t))}{n(n-1)}=2\rho(t)(1-\rho(t))+2(1-2\rho(t))\frac{X_n(t)}{\sqrt{n}}+O(\frac{X_n(t)^2}{n}).
\]

Whenever this event is fulfilled, the order of the giant component is increased by the order of the connected component of the `new' vertex. It is known that the order of such a component is approximately distributed as the total progeny of a  Bienaymé-Galton-Watson (BGW for short) tree with $\Bin(n-L_n(t), \frac tn)$ reproduction distribution. We denote by $\mu_n(t)$ this distribution. The limiting expectation of $\mu_n(t)$ is
\begin{equation}
\label{lambda}
\lambda(t):=t(1-\rho(t)).
\end{equation}

Let us note that $\lambda(t)$ verifies $-\lambda(t)e^{-\lambda(t)}=-te^{-t}$ and thus $-\lambda(t)=W_0(-te^{-t})$ where $W_0$ is the principal branch of the Lambert W-function. In particular, $\lambda(1)=1$, the function $\lambda$ is smaller than 1 and bounded away from 1 on the interval $[t_0,t_1]$.
 
Summing over all the added edges between $t$ and $t+h$ we have the following representation:
\begin{equation}\label{sumheur}
L_n(t+h)-L_n(t)=\sum_{i=1}^{ N_n(t)} \epsilon_i T_i,
\end{equation}
where $\epsilon_i$ are Bernoulli$(\zeta_n(t))$ random variables that are equal to 1 when the associated edge connects the giant component to a vertex outside of it, and the random variables $T_i$ are $\mu_n(t)$ distributed and represent the order of  components outside of $\cm(t)$. 

The expectation of \eqref{sumheur} yields
\begin{multline*}
\E(N_n)\E(\epsilon)\E(T)=\frac{nh}{2}\left(\frac{2\rho(t)(1-\rho(t))}{1-\lambda(t)}+2\left(\frac{1-2\rho(t)}{1-\lambda(t)}
-\frac{\rho(t)(1-\rho(t))t}{(1-\lambda(t))^2}\right)\frac{X_n(t)}{\sqrt{n}}\right)\\
+O\left(ht^2 X_n(t)^{2}\right).
\end{multline*}

Note that the first order term $nh\frac{\rho(t)(1-\rho(t))}{1-\lambda(t)}$ coincides with $nh\rho'(t)$ and therefore will compensate the increment of $n\rho(t)$ in the variation of $L_n(t)-n\rho(t)=\sqrt{n}X_n(t)$. The second term $\sqrt{n}h\left(\frac{1-2\rho(t)}{1-\lambda(t)}
-\frac{\rho(t)(1-\rho(t))t}{(1-\lambda(t))^2}\right){X_n(t)}$ will give the drift term in the limiting stochastic differential equation.
\medskip

We are now interested in the fluctuations in \eqref{sumheur}. They come from the three sources of randomness $N_n(t), \epsilon_i, T_i$ which will turn out to be asymptotically independent. Informally, we can describe these fluctuations with the help of three standard Brownian motions $C^{(N)}$, $C^{(T)}$ and  $C^{(\epsilon)}$:
\begin{itemize}
\item fluctuations of $N_n(t)$: $\sqrt{n/2} \E(\epsilon T)dC^{(N)}(t)= \sqrt{2n}\frac{\rho(t)(1-\rho(t))}{1-\lambda(t)}dC^{(N)}(t)$, 
\item fluctuations due to $\epsilon_i$: 
\[\sqrt{\var(\epsilon)\E(N_n(t))\E(T)^2}dC^{(\epsilon)}(t)=\sqrt{n\rho(t)(1-\rho(t))\frac{1-2\rho(t)(1-\rho(t))}{(1-\lambda(t))^2}}dC^{(\epsilon)}(t),\]
\item fluctuations due to $T_i$: \[\sqrt{\E(N_n(t)\epsilon)\var(T)}dC^{(T)}(t)=\sqrt{n\rho(t)(1-\rho(t))\frac{\lambda(t)}{(1-\lambda(t))^3}}dC^{(T)}(t).\]
\end{itemize}
Therefore, the limiting equation for $(X_n(t))_t$ is: 
\begin{align*}
dX(t)=&\left(\frac{1-2\rho(t)}{1-\lambda(t)} -\frac{\rho(t)(1-\rho(t))t}{(1-\lambda(t))^2}\right)X(t) dt + \sqrt{2}\frac{\rho(t)(1-\rho(t))}{1-\lambda(t)}dC^{(N)}(t) \\ &+\sqrt{\rho(t)(1-\rho(t))\frac{1-2\rho(t)(1-\rho(t))}{(1-\lambda(t))^2}}dC^{(\epsilon)}(t)
+ \sqrt{\frac{\rho(t)(1-\rho(t))\lambda(t)}{(1-\lambda(t))^3}}dC^{(T)}(t).
\end{align*}

Assuming the independence of the three above Brownian motions, we can regroup them under a single Brownian motion $C$ and get 
\begin{equation}\label{SDE}
dX(t)= \left( \frac{1-2\rho(t)}{1-\lambda(t)} -\frac{\rho(t)(1-\rho(t))t}{(1-\lambda(t))^2} \right)    X(t) dt +\sqrt{\frac{\rho(t)(1-\rho(t))}{(1-\lambda(t))^{3}}}dC(t).
\end{equation}

Solving this equation gives
\begin{equation}
\label{Xlimit}
X(t)=\frac{1-\rho(t)}{1-\lambda(t)}B\left(\frac{\rho(t)}{1-\rho(t)}\right),
\end{equation}
where $B$ is a standard Brownian motion.

\section{A representation of the infinitesimal evolution of the giant component \label{representation}}

Let $h>0$. The goal of this section is to give a representation of the variation of $L_n(t)$ between $t$ and $t+h$, in order to treat this variation as a sum of i.i.d. random variables, up to a tractable error.

First note that, conditional on the situation at time $t$, $L_n(t+h)$ only depends on the edges added between $t$ and $t+h$, and not on the order in which we add them. Therefore, we can add them in a strategic order. Call $\cG^{e}(t)$ the subgraph induced in $\cG(t)$ by vertices outside of $\cm(t)$. 
We partition the edges added between $t$ and $t+h$ in three categories:
\begin{itemize}
 \item the edges which connect vertices of $\cm(t)$, they do not affect $L_n(t)$ and can be ignored;
\item the edges that connect vertices of $\cG^{e}(t)$;

\item  the edges that connect $\cm(t)$ to $\cG^{e}(t)$.
\end{itemize}

Note that there is not a complete independence between the connections in $\cm(t)$ and $\cG^{e}(t)$, since the orders of the connected components of $\cG^{e}(t)$ have to be smaller than $|\cm(t)|$. To overcome this difficulty, we introduce two graphs, denoted $\cGt(t)$ and $\cGt(t,h)$, for which the pair of vertices outside the giant component are connected independently with probability respectively $\frac{t}{n}$ and $\frac{t+h}{n}$. 

First, we construct $\cGt(t)$ on $\enu$, having $\cm(t)$ as a connected component in the following way:

\begin{itemize}
\item every edge between vertices inside $\cm(t)$
is present in $\cGt(t)$ if and only if it is present in $\cG(t)$;
\item $\cGt(t)$ has no edge between a vertex in $\cm(t)$ and another outside;
\item an edge between two vertices outside $\cm(t)$, is added with probability $\frac{t}{n}$ independently from every other edge. We call them \textit{external} edges.
\end{itemize}

Secondly, the graph $\cGt(t,h)$ is defined from $\cGt(t)$ by adding independently  with probability $\frac{\frac{h}{n}}{1-\frac{t}{n}}=\frac{h}{n-t}$ all the edges that are not in $\cGt(t)$. The steps of the construction of $\cGt(t,h)$ from $\cG(t)$ are illustrated in Figures~\ref{graph_g}, \ref{graph_gt} and \ref{graph_gth}. 
Let us call a {\it bridge}, an edge in $\cGt(t,h)$ between a vertex in $\cm(t)$ and another one outside of $\cm(t)$. Let $Y_n$ denote their number. Let $\cGt^{e}(t,h)$ be the subgraph induced in $\cGt(t,h)$ by vertices which are not in $\cm(t)$.    It is obvious from the construction that $\cGt^{e}(t,h)$ is distributed as an  $\ER(n-L_n(t),\frac{t+h}{n})$-random graph and that  $Y_n$ is a $\Bin(L_n(t)(n-L_n(t)), \frac{h}{n-t})$ random variable. 

Let $\Lt(t,h)$ be the order of the component of $\cGt(t,h)$ containing $\cm(t)$. We will see that $\Lt(t,h)$ behaves similarly to $\Lt(t+h)$.

An analog construction was made in \cite{Behrisch}, in the setting of random uniform hypergraphs in order to prove that the order of the giant component follows a local limit theorem at some {\it fixed} time. 
We are going to use a modified version of Lemma 10 of \cite{Behrisch}, namely dealing with estimations that are valid uniformly on a time interval. This will make it possible to reduce the study of $\Delta L_n(t):=L_n(t+h)-L_n(t)$ to the study of $\Deltat (t):=\Lt(t,h)-L_n(t)$, on the event
\[\mathcal{E}_n(t):= \{|L_n(t)-n\rho(t)|\leq n^{0.7}\}=\{|X_n(t)|\leq n^{0.2}\}.\]

We state it in our framework:
\begin{lem}[Lemma 10 of \cite{Behrisch}]\label{lem-behrisch}
There is a number $n_0>0$ and a constant $C_1$ such that, for $t\in[t_0,t_1],$ for all integers $n>n_0$, all integers $n_1$ such that $|n_1-\rho(t)n|<n^{0.7},$ and all integers $\nu$ we have
\[|\P(\Delta L_n(t)=\nu|L_n(t)=n_1)-\P(\Deltat(t)=\nu|L_n(t)=n_1)|\le C_1 n^{-10}.\]
As a consequence,
$|\E[(\Delta L_n(t))^k\mid L_n(t)=n_1]-\E[(\Deltat(t))^k|L_n(t)=n_1]|\leq C_1 n^{k-9}$ for $k\in\{1,2,3\}$.
\end{lem}

The only difference between this statement and the one presented in \cite{Behrisch} is the fact that the bound is uniform in $t\in[t_0,t_1]$.
To get this statement, the proof in \cite{Behrisch} is still valid; the only missing ingredient is a uniform (in $t$) upper bound for the probability of existence of components of order larger than $(\ln n)^2$ outside of the giant one. We show this bound in  Appendix \ref{rareevent} (Lemma \ref{boundan}).

For a fixed time $t>1$, it is known that the complement of $\mc{E}_n(t)$ denoted by $\mc{E}^{c}_n(t)$ appears rarely: $\P(\mathcal{E}^{c}_n(t))=O(\exp(-Cn^{0.4}))$ (see for example Puhalskii \cite{Puhalskii}, Theorem 2.3). However  we need a bound of $\P(\mathcal{E}^{c}_n(t))$ that is uniform in $t\in[t_0,t_1]$. To the best of our knowledge such a uniform bound has never been shown so we shall prove in Lemma \ref{bounden} the following uniform estimate:
\begin{equation}
\label{complen}
\text{Uniformly on } t\in[t_0,t_1],\quad 
\P(\mathcal{E}^{c}_n(t))=O(\exp(-Cn^{0.4})).
\end{equation}
The proof relies on some results by Stepanov~\cite{Stepanov}. \medskip

To study $\Deltat(t)$, we give a description of the edges that contribute to it.
We call $x_1, x_2,\ldots, x_{Y_n}$ the sequence (in an arbitrary order) of the endpoints of the bridges in $\cGt(t,h)$ which are outside of $\cm(t)$. We denote by  $C_{x_1}, \ldots, C_{x_{Y_n}}$ their connected components in $\cGt^{e}(t,h)$.

We have to take into account the unlikely but possible event that two bridges point to the same component of $\cGt^{e}(t,h)$. We introduce the set \[\mathcal{R}:=\{2\le j\le Y_n \text{ s.t. } C_{x_j}=C_{x_i} \text{ for some } i<j \}\]
and define
\begin{equation}
C^{\star}_{x_i}:=\begin{cases}
\emptyset &\text{ if } i\in\mathcal{R}\\
C_{x_i}& \text{ otherwise}.
\end{cases}
\end{equation}
(for the graph depicted in Figure~\ref{graph_gth}, $\mathcal{R}=\{3\}$)

The variables $|C_{x_i}|$ are classically compared to the total progeny of a Bienyaimé-Galton-Watson process with a binomial reproduction distribution. The purpose of the following lemma is to use this comparison in order to give properties of $\Deltat(t)$ that will lead to upper and lower bounds of the moments of  $\Deltat(t)$ in the next section.

\begin{lem}\label{encadrement}
Let $\BGW(n,p)$ denote a Bienaymé-Galton-Watson process with $\Bin(n,p)$ reproduction distribution. With the previously introduced notation, 
\[
\Deltat(t)=\sum_{\substack{i=1,\\ i\not\in \mathcal{R}}}^{Y_n}|C_{x_i}|=\sum_{i=1}^{Y_n} |C^{\star}_{x_i}|.
\]
\begin{enumerate}
\item 
There exists a sequence $(M_i)_{i\in\NN^*}$ of i.i.d. random variables independent of $Y_n$ such that
\begin{enumerate}
\item \label{encadrement:va}
 $\forall 1\leq j\leq Y_n, \quad \sum_{i=1}^j|C_{x_i}^{\star}|\leq \sum_{i=1}^j M_i$,
 \item \label{encadrement:M} the distribution of $M_i$ conditional on $\Lt(t)$ is the same as the total progeny of a $\BGW(n-\Lt(t),\frac{t+h}{n})$;
\end{enumerate} 
we write accordingly
\[\Deltat(t)\leq\min\left(\sum_{i=1}^{Y_n}M_i,n\right).\]
\item
Moreover, 
\begin{enumerate}
\item \label{encadrement:Y} the conditional distribution of $Y_n$ given $\Lt(t)$ is $\Bin(\Lt(t)(n-\Lt(t)),\frac{h}{n-t})$;
\item \label{encadrement:C} the conditional distribution of $|C_{x_i}|$ given $(Y_n,\Lt(t))$ dominates the conditional distribution of the minimum of the total progeny of ${\BGW(n-\Lt(t)-\ln^2(n),\frac{t+h}{n})}$ and $\ln^2(n)$;
\item \label{encadrement:A} Let $\mathcal{A}_n(t):=\{\text{no component of }\cGt^{e}(t,h) \text{ has size greater than } \ln^2(n)\}.$ \par
\noindent The conditional distribution of $|\mathcal{R}|\un_{\mathcal{A}_n(t)}$ 
given $(Y_n,\Lt(t))$ is dominated by the binomial distribution $\Bin\left(Y_n,\frac{Y_n\ln^2(n)}{n-\Lt(t)}\right)$. 
\end{enumerate}
\end{enumerate}
\end{lem}

\begin{rem}\
\begin{itemize}
\item The general idea behind this lemma is that, up to negligible effects, $\Deltat(t)$ is close to a sum of $Y_n$ i.i.d. random variables $M_i$. The difference between $\Deltat(t)$ and $\sum_{i=1}^{Y_n}M_i$ comes from two sources: the indices in $\mathcal{R}$ and the  difference between $M_i$ and $|{C}_{x_i}|$. Using the last two points of the lemma, we will see in the next section that this difference is very small.

\item 
We will show that, uniformly on $t\in[t_0,t_1]$,
 on the event $\mathcal{E}_n(t)$,
\[\P(\mathcal{A}^{c}_n(t)\mid \sigma(X_n(s), s\leq t))=O(n^{-10}).\]
This is a consequence of usual bounds on the order of the components of a subcritical Erd\H{o}s-R\'enyi graph. A detailed proof is given in Appendix \ref{rareevent}.
\end{itemize}
\end{rem}

Before starting the proof of Lemma \ref{encadrement}, let us recall some classical inequalities between the order of a component and the total progeny of a Binomial BGW process, derived from the exploration of a component in a breadth-first order.
 \begin{tref}[Theorems 4.2 and 4.3 in \cite{vdh}]
 \label{propvdh} 
 For a positive integer $n$ and a real $p\in]0,1[$, let $L_{n,p}$ denote the order of the component of a vertex in $\ER(n,p)$ and 
  let $T_{n,p}$ denote the total progeny of a $\BGW(n,p)$. Then for every $k\in\{0,1,\ldots,n\}$,
  \begin{eqnarray}
   &&  \P(L_{n,p}\geq k)\leq \P(T_{n,p}\geq k),\label{vdhmajcomp}\\
   &&  \P(L_{n,p}\geq k)\geq \P(T_{n-k,p}\geq k). \label{vdhmincomp}
 \end{eqnarray}
 \end{tref}
\begin{proof}[Proof of Lemma \ref{encadrement}]

To prove Assertions \ref{encadrement:va} and \ref{encadrement:M}, we define recursively a sequence 
of graphs $(\cGt^{e,i}(t,h))_i$ and a sequence $(\tilde{C}_{x_i})_i$ such that $\tilde{C}_{x_i}$ is a component of $x_i$ in $\cGt^{e,i-1}(t,h)$ and \[\forall j\geq 1, \quad \bigcup_{i=1}^j C^\star_{x_i}\subset \bigcup_{i=1}^j\tilde{C}_{x_i}.\]

We do it as follows\footnote{See Figures~\ref{graph_gtone} and \ref{graph_gtfour} for a realization of the sequence $(\cGt^{e,i}(t,h))_i$ for the graph $\cGt(t,h)$ depicted in Figure~\ref{graph_gth}.}. 

First set $\tilde{C}_{x_1}=C^{\star}_{x_1}$. Then call $\cGt^{e,1}(t,h)$ the graph obtained from $\cGt^{e}(t,h)$ by redrawing every edge containing a vertex of $C^{\star}_{x_1}$. More precisely,  we remove all those edges and then add each pair of vertices containing an element of $C^{\star}_{x_1}$ with probability $\frac{t+h}{n}$ independently of anything else. Note that $\cGt^{e,1}(t,h)$ is independent of $C^{\star}_{x_1}$.

We now call $\tilde{C}_{x_2}$ the component of $x_2$ in $\cGt^{e,1}(t,h)$. It follows from the construction that $\tilde{C}_{x_1}$ and $\tilde{C}_{x_2}$ are independent and have the same distribution. 

We also note that 
\[C^\star_{x_2}\subset \tilde{C}_{x_2}. \]
Indeed, either $C^{\star}_{x_2}=\emptyset$ and the statement is trivial or $x_2$ and $x_1$ are not in the same component in $\cGt^{e}(t,h)$, which means that
when redrawing the edges we did not change edges between vertices in $C_{x_2}$. It is however possible that by redrawing the edges, we add edges connecting for instance $C_{x_2}$ to $C_{x_1}$ in $\cGt^{e,1}(t,h)$. Observe that it may occur that this redrawing creates also paths  between $C_{x_2}$ and $C_{x_i}$ for some $i\geq3$,  via some intermediate vertex belonging to $C_{x_1}$ (see Figure~\ref{graph_gtone}).

For $i\geq2$, 

\begin{itemize}
    \item if $x_{i+1}\in \tilde{C}_{x_l}$ for some $l\leq i$, then $C^{\star}_{x_{i+1}}\subset \tilde{C}_{x_l}$. We set $\tilde{C}_{x_{i+1}}:=\emptyset$ and define $\cGt^{e,i}(t,h)$ as $\cGt^{e,i-1}(t,h)$.
    \item if $x_{i+1}\notin \cup_{l=1}^{i}\tilde{C}_{x_l}$, then 
    $C^{\star}_{x_{i+1}}$ is disjoint of $\cup_{l=1}^{i}\tilde{C}_{x_l}$. We  redraw the edges having one extreme point in $\cup_{l=1}^{i}\tilde{C}_{x_l}$. This defines a new graph denoted by $\cGt^{e,i}(t,h)$. This graph is again distributed as an $\ER (n-L_n(t), \frac{t+h}{n})$ and is independent of $\cup_{l=1}^{i}\tilde{C}_{x_l}$.
    We now call $\tilde{C}_{x_{i+1}}$ the component of $x_{i+1}$ in $\cGt^{e,i}(t,h)$.
\end{itemize}

\definecolor{orange}{rgb}{1.,0.6,0.} 
\definecolor{rouge}{rgb}{0.8,0.,0.} 
\definecolor{bleuc}{rgb}{0.2,0.2,1.} 
\definecolor{vert}{rgb}{0.,0.6,0.} 
\definecolor{bleu}{rgb}{0.,0.,1.} 
\definecolor{gris}{rgb}{0.7,0.7,0.7} 
\tikzstyle{fleche}=[->,>=latex,thin]

\newcommand{\compmax}[1][gris]
{\draw [rotate around={88.77591747782445:(-4.54,1.48)},line width=0.4pt,fill=gris!10] (-4.54,1.48) ellipse (2.7439921570113612cm and 1.4322684656655273cm);
\draw [line width=1.2pt,color=#1] (-5.19,3.38)-- (-4.63,2.78);
\draw [line width=1.2pt,color=#1] (-4.63,2.78)-- (-4.91,2.14);
\draw [line width=1.2pt,color=#1] (-5.49,2.6)-- (-4.63,2.78);
\draw [line width=1.2pt,color=#1] (-5.49,2.6)-- (-5.55,1.84);
\draw [line width=1.2pt,color=#1] (-5.55,1.84)-- (-5.65,1.1);
\draw [line width=1.2pt,color=#1] (-4.81,1.28)-- (-4.91,2.14);
\draw [line width=1.2pt,color=#1] (-4.91,2.14)-- (-4.21,1.84);
\draw [line width=1.2pt,color=#1] (-5.45,0.32)-- (-4.97,-0.36);
\draw [line width=1.2pt,color=#1] (-5.45,0.32)-- (-4.71,0.44);
\draw [line width=1.2pt,color=#1] (-5.45,0.32)-- (-4.81,1.28);
\draw [line width=1.2pt,color=#1] (-4.71,0.44)-- (-4.13,-0.22);
\draw [line width=1.2pt,color=#1] (-5.19,3.38)-- (-4.39,3.92);
\draw [line width=1.2pt,color=#1] (-5.19,3.38)-- (-4.49,3.25);
\draw [line width=1.2pt,color=#1] (-3.87,3.14)-- (-3.77,2.42);
\draw [line width=1.2pt,color=#1] (-4.21,1.84)-- (-3.77,2.42);
\draw [line width=1.2pt,color=#1] (-4.21,1.84)-- (-3.51,1.44);
\draw [line width=1.2pt,color=#1] (-4.21,1.84)-- (-4.63,2.78);
\draw [line width=1.2pt,color=#1] (-4.09,1.)-- (-3.51,1.44);
\draw [line width=1.2pt,color=#1] (-4.71,0.44)-- (-3.57,0.38);
\draw [line width=1.2pt,color=#1] (-4.13,-0.22)-- (-4.59,-0.86);
\draw [line width=1.2pt,color=#1] (-4.97,-0.36)-- (-4.59,-0.86);
\draw [color=#1] (-4.59,-0.86) circle (1.5pt);
\draw [color=#1] (-3.87,3.14) circle (1.5pt);
\draw [color=#1] (-3.51,1.44) circle (1.5pt);
\draw [color=#1] (-3.57,0.38) circle (1.5pt);
\draw [color=#1] (-3.77,2.42) circle (1.5pt);
\draw [color=#1] (-4.49,3.25) circle (1.5pt);
\draw [color=#1] (-4.39,3.92) circle (1.5pt);
\draw [color=#1] (-5.19,3.38) circle (1.5pt);
\draw [color=#1] (-4.63,2.78) circle (1.5pt);
\draw [color=#1] (-5.49,2.6) circle (1.5pt);
\draw [color=#1] (-5.55,1.84) circle (1.5pt);
\draw [color=#1] (-4.91,2.14) circle (1.5pt);
\draw [color=#1] (-5.65,1.1) circle (1.5pt);
\draw [color=#1] (-4.81,1.28) circle (1.5pt);
\draw [color=#1] (-5.45,0.32) circle (1.5pt);
\draw [color=#1] (-4.71,0.44) circle (1.5pt);
\draw [color=#1] (-4.97,-0.36) circle (1.5pt);
\draw [color=#1] (-4.13,-0.22) circle (1.5pt);
\draw [color=#1] (-4.09,1.) circle (1.5pt);
\draw [color=#1] (-4.21,1.84) circle (1.5pt);
}

\newcommand{\cmleg}
{\draw (-6.4,4.54) node {$\cm(t)$};
\draw [fleche,line width=1.2pt,dotted,color=black] (-6.57,4.28) -- (-5.674952470825097,3.089828817433463);}
\newcommand{\addinedge}
{\draw [line width=1.2pt,dash pattern=on 5pt off 5pt, color=black] (-4.39,3.92)-- (-3.87,3.14);
\draw [line width=1.2pt,dash pattern=on 5pt off 5pt, color=black] (-4.13,-0.22)-- (-3.57,0.38);}
\newcommand{\outnodes}
{
\draw [fill=black] (-3.15,4.6) circle (1.5pt);
\draw [fill=black] (-2.21,5.08) circle (1.5pt);
\draw [fill=black] (-2.09,4.08) circle (1.5pt);

\draw [fill=black] (-1.69,2.6) circle (1.5pt);
\draw [fill=black] (-0.11,3.3) circle (1.5pt);
\draw [fill=black] (-1.99,1.72) circle (1.5pt);

\draw [fill=black] (-0.45,0.96) circle (1.5pt);
\draw [fill=black] (-0.53,1.78) circle (1.5pt);
\draw [fill=black] (-1.25,0.48) circle (1.5pt);

\draw [fill=black] (-2.87,-0.78) circle (1.5pt);
\draw [fill=black] (-1.81,-0.88) circle (1.5pt);
\draw [fill=black] (-0.83,-0.94) circle (1.5pt);
\draw [fill=black] (0.59,1.) circle (1.5pt);
\draw [fill=black] (-2.65,3.) circle (1.5pt);

\draw [fill=black] (-1.0296229002399726,4.06036338704148) circle (1.5pt);
\draw [fill=black] (-3.33,-1.6) circle (1.5pt);

}
\newcommand\inbridges[1][gris]{
\draw[color=#1] (-4.18,-0.8) node {$a_5$};
\draw[color=#1] (-3.73,3.43) node {$a_1$};
\draw[color=#1] (-3.37,1.7) node {$a_3$};
\draw[color=#1] (-3.43,0.63) node {$a_4$};
\draw[color=#1] (-3.5,2.3) node {$a_2$};
}
\newcommand\outbridges{
\draw[color=black] (-1.8,4.2) node {$x_1$};
\draw[color=black] (-2.5,2.75) node {$x_2$};
\draw[color=black] (-2.2,2) node {$x_3$};
\draw[color=black] (-0.95,0.4) node {$x_{4}$};
\draw[color=black] (-3,-1.7) node {$x_{5}$};
}
\newcommand{\bridges}{
\draw [line width=1.2pt,dash pattern=on 5pt off 5pt,color=vert] (-2.09,4.08)-- (-3.87,3.14);
\draw [line width=1.2pt,dash pattern=on 5pt off 5pt,color=vert] (-1.99,1.72)-- (-3.51,1.44);
\draw [line width=1.2pt,dash pattern=on 5pt off 5pt,color=vert] (-1.25,0.48)-- (-3.57,0.38);
\draw [line width=1.2pt,dash pattern=on 5pt off 5pt,color=vert] (-2.65,3.)-- (-3.77,2.42);
\draw [line width=1.2pt,dash pattern=on 5pt off 5pt,color=vert] (-4.59,-0.86)-- (-3.33,-1.6);
}
\newcommand{\outedges}{
\draw [line width=1.2pt] (-3.15,4.6)-- (-2.09,4.08);
\draw [line width=1.2pt] (-2.09,4.08)-- (-2.21,5.08);
\draw [line width=1.2pt] (-1.99,1.72)-- (-1.69,2.6);
\draw [line width=1.2pt] (-0.53,1.78)-- (-1.99,1.72);
\draw [line width=1.2pt] (0.59,1.)-- (-0.45,0.96);
\draw [line width=1.2pt] (-0.45,0.96)-- (-1.25,0.48);
\draw [line width=1.2pt] (-1.81,-0.88)-- (-2.87,-0.78);
\draw [line width=1.2pt] (-2.65,3.)-- (-1.69,2.6);
}
\newcommand{\outedgesone}{
\draw [line width=1.2pt,color=bleu] (-3.15,4.6)-- (-2.21,5.08);
\draw [line width=1.2pt,color=bleu] (-2.09,4.08)-- (-2.21,5.08);
\draw [line width=1.2pt,color=bleu] (-2.09,4.08)-- (-1.01,4.06);
}
\newcommand{\outedgestwo}{
\draw [line width=1.2pt,color=bleu] (-1.99,1.72)-- (-1.69,2.6);
\draw [line width=1.2pt,color=bleu] (-0.53,1.78)-- (-0.45,0.96);
\draw [line width=1.2pt,color=bleu] (0.59,1.)-- (-0.45,0.96);
\draw [line width=1.2pt,color=bleu] (-0.45,0.96)-- (-1.25,0.48);
\draw [line width=1.2pt,color=bleu] (-2.65,3.)--(-1.69,2.6);
}
\newcommand{\outedgesthree}{
\draw [line width=1.2pt,dash pattern=on 5pt off 5pt,color=bleu] (-1.81,-0.88)-- (-0.83,-0.94);
\draw [line width=1.2pt,dash pattern=on 5pt off 5pt,color=bleu] (-2.87,-0.78)-- (-3.33,-1.6);
}
\newcommand{\extcompone}{
\draw [rotate around={-21.551937422042194:(-2.22,4.49)},line width=0.4pt,color=bleuc] (-2.22,4.49) ellipse (1.4430275304103777cm and 0.6673293441189713cm);
\draw [fleche,line width=1.2pt,dotted,color=bleuc] (-0.01,5.22) -- (-1.3340918374866606,4.760590794577469);
\draw (0.4,5.25) node {$C_{x_1}$};
}
\newcommand{\extcomptone}{
\draw [rotate around={-21.551937422042194:(-2.22,4.49)},line width=0.4pt,color=bleuc] (-2.22,4.49) ellipse (1.4430275304103777cm and 0.6673293441189713cm);
\draw [fleche,line width=1.2pt,dotted,color=bleuc] (-0.01,5.22) -- (-1.3340918374866606,4.760590794577469);
\draw (0.4,5.25) node {$\tilde{C}_{x_1}$};
}
\newcommand{\extcomptwo}{
\draw [rotate around={-55.244087446458686:(-2.21,2.39)},line width=0.4pt,color=bleuc] (-2.21,2.39) ellipse (0.8713795063628349cm and 0.63529697316227cm);
\draw [fleche,line width=1.2pt,dotted,color=bleuc] (1.49,3.26) -- (-1.50992575179117,2.347240831414552);
\draw (2.4,3.3) node {$C_{x_2}=C_{x_3}$};
\draw [rotate around={37.69424046668924:(-0.43,0.96)},line width=0.4pt,color=bleuc] (-0.43,0.96) ellipse (1.1690713336985008cm and 1.0283616986622908cm);
\draw [fleche,line width=1.2pt,dotted,color=bleuc] (1.87,1.88) -- (0.61,1.48);
\draw (2.2,1.95) node {$C_{x_4}$};
}
\newcommand{\extcompfive}{
\draw [rotate around={45.70731936854426:(-3.11,-1.23)},line width=0.4pt,color=bleuc] (-3.11,-1.23) ellipse (0.6844736550228593cm and 0.3747054635581775cm);
\draw [fleche,line width=1.2pt,dotted,color=bleuc] (-1.43,-2.76) -- (-2.85,-1.5);
\draw (-1.,-2.8) node {$C_{x_5}$};
}
\begin{figure}[tb]
\begin{minipage}[t]{\textwidth}
\begin{center}
\scalebox{0.5}{
\begin{tikzpicture}[line cap=round,line join=round,>=triangle 45,x=1.0cm,y=1.0cm]
\clip(-7.03,-2.7) rectangle (3.65,6.46);
\draw [line width=0.4pt] (-2.61,1.56) circle (4.208800304124681cm);
\compmax[black];
\cmleg
\outnodes;
\outedges;
\end{tikzpicture}
}
\caption{ A realization of graph $\cG(t)$ for which $L_n(t)=20$.} 
\label{graph_g}
\end{center}
\end{minipage}\\
\begin{minipage}[t]{0.48\textwidth}
\scalebox{0.5}{
\begin{tikzpicture}[line cap=round,line join=round,>=triangle 45,x=1.0cm,y=1.0cm]
\clip(-7.03,-3) rectangle (3.65,6.46);
\draw [line width=0.4pt] (-2.61,1.56) circle (4.208800304124681cm);
\compmax[black];
\cmleg;
\outnodes;
\outedgesone;
\outedgestwo;
\end{tikzpicture}
}
\caption{ A realization of graph $\cGt(t)$ obtained from $\cG(t)$ by resampling independently with probability $\frac{t}{n}$, all edges `external' to $\cm(t)$.} 
\label{graph_gt}
\end{minipage}\quad
\begin{minipage}[t]{0.48\textwidth}
\scalebox{0.5}{
\begin{tikzpicture}[line cap=round,line join=round,>=triangle 45,x=1.0cm,y=1.0cm]
\clip(-7.03,-3) rectangle (3.65,6.46);
\draw [line width=0.4pt] (-2.61,1.56) circle (4.208800304124681cm);
\compmax[black];
\addinedge
\outnodes;
\bridges;\inbridges[black];\outbridges[black];
\outedgesone;\outedgestwo;\outedgesthree;
\extcompone;\extcomptwo;\extcompfive;
\end{tikzpicture}
}
\caption{Graph $\cGt(t,h)$ obtained from  $\cGt(t)$ by adding edges (drawn as dashed lines) independently with probability $\frac{h}{n-t}$. In this realization,  there are ${Y_n=5}$ bridges drawn as green dashed lines and nine external edges drawn in blue. Its largest component is the union of $\cm(t)$ and $\cup_{i=1}^{5}C_{x_i}$. } 
\label{graph_gth}
\end{minipage}
\begin{minipage}[t]{0.48\textwidth}
\scalebox{0.5}{
\begin{tikzpicture}[line cap=round,line join=round,>=triangle 45,x=1.0cm,y=1.0cm]
\clip(-7.03,-3) rectangle (3.65,6.46);
\draw [rotate around={154.87136048572663:(-0.7988937554054354,2.158182357464605)},line width=0.4pt,color=rouge,fill=rouge!5] (-0.7988937554054354,2.158182357464605) ellipse (2.1870699374959908cm and 1.987983698894146cm);
\draw [line width=0.4pt] (-2.61,1.56) circle (4.208800304124681cm);
\outnodes;
\outbridges;
\outedgestwo;
\outedgesthree;
\extcomptone;\extcomptwo;\extcompfive;
\draw [line width=1.2pt,color=rouge] (-1.0296229002399726,4.06036338704148)-- (-0.11,3.3);
\draw [line width=1.2pt,color=rouge] (-2.09,4.08)-- (-3.15,4.6);
\draw [line width=1.2pt,color=rouge] (-1.69,2.6)-- (-1.0296229002399726,4.06036338704148);
\draw [line width=1.2pt,color=rouge] (-0.53,1.78)-- (-1.0296229002399726,4.06036338704148);
\draw [fleche,line width=1.2pt,dotted,color=rouge] (1.25,4.42) -- (0.07295566625871719,3.944833079584321);
\draw (1.52,4.68) node {$\tilde{C}_{x_2}$};
\end{tikzpicture}
}
\caption{ Subgraph $\cGt^{e,1}(t,h)$. New edges with respect to $\cGt^{e}(t,h)$ are drawn in red and obtained from an independent resampling of edges having an endpoint in $\tilde{C}_{x_1}$, with probability $\frac{t+h}{n}$. In this example, vertices $x_3$ and $x_4$ are in $\tilde{C}_{x_2}$. Therefore, $\tilde{C}_{x_3}=\tilde{C}_{x_4}=\emptyset$ and $\cGt^{e,3}(t,h)=\cGt^{e,2}(t,h)=\cGt^{e,1}(t,h)$.  } 
\label{graph_gtone}
\end{minipage}\quad
\begin{minipage}[t]{0.48\textwidth}
\scalebox{0.5}{
\begin{tikzpicture}[line cap=round,line join=round,>=triangle 45,x=1.0cm,y=1.0cm]
\clip(-7.03,-3) rectangle (3.65,6.46);
\draw [rotate around={68.19859051364817:(-2.63,0.08)},line width=0.4pt,color=orange,fill=orange!5] (-2.63,0.08) ellipse (2.173351839802012cm and 0.2888913629209182cm);
\draw [line width=0.4pt] (-2.61,1.56) circle (4.208800304124681cm);
\outnodes;
\outbridges;
\outedgesthree;
\extcomptone;\extcompfive;
\draw [rotate around={154.87136048572663:(-0.7988937554054354,2.158182357464605)},line width=0.4pt,color=rouge] (-0.7988937554054354,2.158182357464605) ellipse (2.1870699374959908cm and 1.987983698894146cm);
\draw [line width=1.2pt,color=orange] (-1.0296229002399726,4.06036338704148)-- (-0.11,3.3);
\draw [line width=1.2pt,color=orange] (-2.09,4.08)-- (-1.0296229002399726,4.06036338704148);
\draw [line width=1.2pt,color=orange] (-1.69,2.6)--(-0.53,1.78) ;
\draw [fleche,line width=1.2pt,dotted,color=rouge] (1.25,4.42) -- (0.07295566625871719,3.944833079584321);
\draw [line width=1.2pt,color=orange] (-0.83,-0.94)-- (0.59,1.);
\draw [line width=1.2pt,color=orange] (-2.87,-0.78)-- (-1.99,1.72);
\draw [line width=1.2pt,color=orange] (-2.65,3.)-- (-3.15,4.6);

\draw (1.52,4.68) node {$\tilde{C}_{x_2}$};
\draw [fleche,line width=1.2pt,dotted,color=orange] (-4.4,-2.7) -- (-3.43,-1.92);
\draw (-4.6,-2.8) node {$\tilde{C}_{x_5}$};
\end{tikzpicture}
}
\caption{ Subgraph $\cGt^{e,4}(t,h)$. New edges with respect to $\cGt^{e,3}(t,h)$ are drawn in orange and obtained from an independent resampling of edges having an endpoint in $\cup_{i=1}^4\tilde{C}_{x_i}$, with probability $\frac{t+h}{n}$. } 
\label{graph_gtfour}
\end{minipage}
\end{figure}

Inequality~\eqref{vdhmajcomp} implies that the conditional distribution of $|\tilde{C}_{x_i}|$ is dominated by the total progeny of a $\BGW(n-L_n(t),\frac{t+h}{n})$. Therefore, there exists a family of i.i.d. random variables $M_i$ independent of $Y_n$, having the same distribution as the total progeny of a $\BGW(n-L_n(t),\frac{t+h}{n})$ such that  $|\tilde{C}_{x_i}|\le M_i$ for every $i$.
 This ends the proof of Assertions \ref{encadrement:va} and \ref{encadrement:M}.
 
 Assertion \ref{encadrement:Y} is a direct consequence of the construction as already mentioned.
\clearpage
 To prove Assertion \ref{encadrement:C}, we first note that inequality \eqref{vdhmincomp} implies that the random variable $\min(L_{n,p},\ln^{2}(n))$ stochastically  dominates 
  $\min(T_{n-\ln^{2}(n),p},\ln^{2}(n))$ since the total progeny of a $\BGW(j,p)$ stochastically dominates the total progeny of a $\BGW(j-1,p)$ for every $j\in\mathbb{N}^*$. Therefore,  the conditional distribution of $|C_{x_i}|$ dominates the conditional distribution of the minimum of $\ln(n)^2$ and of the total progeny of a $\BGW(n-L_n(t)-\ln(n)^2,\frac{t+h}{n})$. 

\medskip 

It remains to prove the domination of $|\mathcal{R}|$ (Assertion \ref{encadrement:A}). Recall that $\mathcal{R}$ is the set of indices $1\le i \le Y_n$ such that $C_{x_i}$ has already been visited by a bridge of index $j<i$, i.e. $x_i\in\bigcup_{j=1}^{i-1}C_{x_j}$. Since the targets are picked uniformly outside $\cm(t)$, the probability of this event is at most 
\[\frac{\sum_{j<i}|C_{x_j}|}{n-L_n(t)}.\]
On the event $\mathcal{A}_n(t)$, this probability is at most 
\[\frac{Y_n(\ln{n})^2}{n-L_n(t)},\] 
which implies \ref{encadrement:A}.
\end{proof}
\section{Convergence of the finite dimensional distributions}
\label{convergencefdd} 
The aim of this section is to prove that $X_n(t):=\frac1{\sqrt n}(L_n(t)-n\rho(t))$ converges in the sense of finite dimensional distributions towards 
\[X(t)=\frac{1-\rho(t)}{1-t(1-\rho(t))}B\left(\frac{\rho(t)}{1-\rho(t)}\right).\]
As already mentioned (see \eqref{SDE} in Section \ref{heuristique}), this process is actually a diffusion with drift $a(t,x):=\left(\frac{1-2\rho(t)}{1-\lambda(t)} -\frac{\rho(t)(1-\rho(t))t}{(1-\lambda(t))^2}\right)x$
and diffusion coefficient $b(t,x):= \frac{{\rho(t)(1-\rho(t))}}{(1-t(1-\rho(t)))^{3}}$.

To show this convergence,  we are going to adapt a theorem by Norman \cite{norman} to a non-compact state space.  In the case of a sequence of compact-valued processes, its  statement makes it possible to reduce the proof of the finite dimensional convergence of a sequence of random processes, to the study of the first three moments of their infinitesimal increments, in the case where the limiting process is a diffusion. The extension of this statement we shall apply is:
\begin{prop}
\label{normanth1}
For each $n\geq 1$, let  $\{Z_{n}(t), t\geq 0\}$ be a stochastic process in $\mathbb{R}$ adapted to a filtration $\{\mathcal{F}_{n}(t),\ t\geq 0\}$.  Assume there exist two $C^3$ functions $\alpha$ and $\beta$, three sequences of nonnegative functions $(e_{1,n})_n$, $(e_{2,n})_n$, $(e_{3,n})_n$ and a sequence $(N_n)_n$ of positive integers converging to $+\infty$,  satisfying, for every $t\geq0$ and every $h>0$:
\begin{xalignat}{2}
\label{hypA1}
 &\mid\E(Z_{n}(t+h)-Z_{n}(t)\mid\mathcal{F}_{n}(t))- h \alpha(Z_{n}(t))\mid \leq e_{1,n}(t,h,Z_{n}(t))\nonumber\\
 &\mid \E((Z_{n}(t+h)-Z_{n}(t))^2\mid \mathcal{F}_{n}(t))- h \beta(Z_{n}(t))\mid \leq e_{2,n}(t,h,Z_{n}(t))\nonumber\\
\text{(Assumption \hypertarget{hypA1}{A1})}\quad&\mid \E(|Z_{n}(t+h)-Z_{n}(t)|^3\mid \mathcal{F}_{n}(t))\mid \leq e_{3,n}(t,h,Z_{n}(t))\nonumber \\
&\forall 0<s<t, \nonumber\\
& \sum_{k=0}^{N_n-1}\sum_{i=1}^{3}\E\left(e_{i,n}\left(s+k\frac{t-s}{N_n},\frac{t-s}{N_n},Z_{n}(s+k\frac{t-s}{N_n})\right)\right)\underset{n\rightarrow +\infty}{\rightarrow} 0. \nonumber
\end{xalignat}
Let $(Z(t))_{t\geq 0}$ be the diffusion process with drift function $\alpha$ and diffusion coefficient $\beta$; denote its semigroup by $\{T_{t},  t\geq 0 \}$. Define the three functions $\epsilon_{1}, \epsilon_{2}, \epsilon_{3}$ by, for every $t\geq 0$ and $h>0$, 
\begin{eqnarray*}
&& \E(Z(t+h)-Z(t)\mid Z(t))= h \alpha(Z(t))+\epsilon_{1}(h,Z(t))\\
&&\E((Z(t+h)-Z(t))^2\mid Z(t))= h \beta(Z(t))+\epsilon_{2}(h,Z(t))\\
&&\E(|Z(t+h)-Z(t)|^3\mid Z(t))= \epsilon_{3}(h,Z(t)).
\end{eqnarray*}
Let $\mc{H}$ denote the set of bounded $C^3$ functions defined on $\mathbb{R}$ whose derivatives are also bounded.
Assume the following properties: 
\begin{eqnarray*}
\text{(Assumption \hypertarget{hypA234}{A2}) } && \forall f\in\mc{H},\,  
T_{t}f\in\mc{H}\text{ and }\forall \tau>0\sup_{0\leq t \leq \tau}\sum_{k=0}^{3}||(T_{t}f)^{(k)}||_{\infty}<+\infty  \\
\text{(Assumption A3) }&&\forall 0<s<t,\, \sum_{k=0}^{N_n-1}\sum_{i=1}^{3}\E\left(\left|\epsilon_i\left(\frac{t-s}{N_n},Z_{n}(s+k\frac{t-s}{N_n})\right)\right|\right)\underset{n\rightarrow +\infty}{\rightarrow} 0\\
\text{(Assumption A4) }&& (Z_{n}(0))_n \text{ converges in distribution to } Z(0).
\end{eqnarray*}
Then for every $k\geq 1$ and for every $0<s_1<\ldots<s_k$, $(Z_{n}(s_1),\ldots,Z_{n}(s_k))_n$ converges in distribution to  $(Z(s_1),\ldots,Z(s_k))$. 
\end{prop}
The proof of this Proposition is postponed to Appendix \ref{annex:proofthnornm}. \medskip

Note that this Proposition is restricted to the case where the limiting process is homogeneous in time, which $X$ is not. To overcome this difficulty, we consider the auxiliary process \[Z_n:=(uX_n)\circ v^{(-1)}\] and use Proposition \ref{normanth1} to show its convergence to a standard Brownian motion.
  In this case, Proposition \ref{normanth1} will apply with diffusion coefficients $\alpha\equiv 0$, $\beta\equiv 1$. This argument is detailed in Section \ref{sub-moments-aux}.
However, before applying this time change we need to evaluate the first three moments of the increments of the original process $X_n(t):=\frac1{\sqrt n}(L_n(t)-n\rho(t))$. This is the goal of Section \ref{sub-moments}.

\subsection{The first three moments of \texorpdfstring{$X_n(t+h)-X_n(t)$}{Xn(t+h)-Xn(t)}}\label{sub-moments}

Set $\Delta X_n(t)=X_n(t+h)-X_n(t)$. 
Let $\{\mathcal{F}_n(t),\, t\geq t_0\}$ denote the filtration generated by the process $\{X_n(t),\, t\geq t_0\}$ for every $n\in\mathbb{N}$. In the sequel, we will use the notation
\[\Pt(\cdot)=\P( \cdot \mid \mathcal{F}_n(t)) \text{ and }\Et(\cdot)=\E( \cdot \mid \mathcal{F}_n(t)).\]
Unless specified, the notation $O(.)$ is always uniform in $t\in [t_0,t_1]$.

\paragraph*{First moment.}
We begin by showing that the expectation of the conditional increment of $X_n(t)$ between $t$ and $t+h$ verifies:
\begin{lem}\label{lem:deltaX}
For $x\in \RR$,  $t>1$,  $h>0$ and $n\in \NN$, set 
\begin{eqnarray*}
a(t,x)&=&\left(\frac{1-2\rho(t)}{1-\lambda(t)} -\frac{\rho(t)(1-\rho(t))t}{(1-\lambda(t))^2}\right)x\\
r_{1,n}(h,x)&= &(hn^{-0.1}+n^{-7.5})+\sqrt{n}\un_{\{|x|> n^{0.2}\}}.
\end{eqnarray*}
There exists a constant $C>0$, such that for $n>t_1$, $h\leq n^{-1}$ and $t\in [t_0,t_1],$ 
\[
\left| \Et(\Delta X_n(t)) - ha(t, X_n(t)) \right| \leq C r_{1,n}(h,X_n(t)). 
\]

\end{lem}
\begin{proof}

Recall that \[\Delta X_n(t)= \frac 1{\sqrt n}\left(\Delta L_n(t)-n(\rho(t+h)-\rho(t))\right).\]
On the event $\mathcal{E}_n(t)$ which belongs to $\mc{F}_n(t),$ Lemma \ref{lem-behrisch} gives  \[\Et[\Delta L_n(t)]=\Et[\Deltat(t)]+O(n^{-8}).\]

We now use Lemma \ref{encadrement} to get an upper bound of $\Et[\Deltat(t)]$. On the event $\mathcal{E}_n(t)$, we denote by $m^{+}_{n,1}$ the conditional first moment of the total progeny of  the subcritical ${\BGW(n-L_n(t),\frac{t+h}{n})}$ given $\mc{F}_t$.   On the event $\mathcal{E}_n(t)$, 
\begin{align}\label{Z+}
\Et[\Deltat(t)]\le \Et\left(\sum_{i=1}^{Y_n}M_i\right)&= \Et(Y_n)m^{+}_{n,1}=L_n(t)(n-L_n(t))\frac{h}{n-t}\frac{1}{1-(n-L_n(t))\frac{t+h}{n}}\nonumber \\
&=n h f_1\left(t+h,\rho(t)+\frac{X_n(t)}{\sqrt{n}}\right)\left(1+O(\frac1n)\right)
 \end{align}
with $f_1(t,x)=\frac{x(1-x)}{1-t(1-x)}$.

This induces the following upper bound for $\Et[\Delta X_n(t)]$ on the event $\mathcal{E}_n(t)$,
\begin{multline*}\Et[\Delta X_n(t)]\le\sqrt{n} \left(h f_1\left(t+h,\rho(t)+\frac{X_n(t)}{\sqrt{n}}\right)(1+O(\frac1n))-(\rho(t+h)-\rho(t))\right)\\+O(n^{-8.5}).
\end{multline*}
Noticing that $\rho'(t)=f_1(t,\rho(t))$, on the event $\mathcal{E}_n(t)$, this upper bound is equal  to 
\begin{eqnarray}\label{mom11}
\Et[\Delta X_n(t)]&\le&
\sqrt{n} h \left(f_1\left(t+h,\rho(t)+\frac{X_n(t)}{\sqrt{n}}\right)-f_1(t,\rho(t))\right)\nonumber\\
&&-\sqrt{n}(\rho(t+h)-\rho(t)-h\rho'(t))+O\left(\frac{h}{\sqrt{n}}+n^{-8.5}\right)\nonumber\\
&=&ha(t,X_n(t))+r^{+}_{1,n}(t,h,X_n(t))
\end{eqnarray}
with 
\begin{eqnarray*}
a(t,X_n(t))&=& X_n(t)\frac{\partial f_1}{\partial x}(t,\rho(t))=\left(\frac{1-2\rho(t)}{1-\lambda(t)} -\frac{\rho(t)(1-\rho(t))t}{(1-\lambda(t))^2}\right)X_n(t)\\
r^{+}_{1,n}(t,h,X_n(t))&=&\sqrt{n} h \left(f_1\left(t+h,\rho(t)+\frac{X_n(t)}{\sqrt{n}}\right)-f_1\left(t,\rho(t)+\frac{X_n(t)}{\sqrt{n}}\right)\right)\\
&&+\sqrt{n} h \left(f_1\left(t,\rho(t)+\frac{X_n(t)}{\sqrt{n}}\right)-f_1(t,\rho(t))-\frac{X_n(t)}{\sqrt{n}}\frac{\partial f_1}{\partial x}(t,\rho(t))\right)\nonumber\\
&&-\sqrt{n}\left(\rho(t+h)-\rho(t)-h\rho'(t)\right)+O\left(\frac{h}{\sqrt{n}}+n^{-8.5}\right)\nonumber\\
&=& O\left(\sqrt n h^2+(1+X^{2}_{n}(t)) \frac{h}{\sqrt n}+n^{-8.5}\right).\nonumber
\end{eqnarray*}

We now want to show that this upper bound is close to the true expectation of $\Delta X_n(t)$ conditional on $\mathcal{F}_n(t)$. The only place where we used an inequality is in \eqref{Z+}. 
Accordingly, we want to bound $\Et\left[\sum_{i=1}^{Y_n}M_i-\Deltat(t)\right]$.
Note that a priori $\sum_{i=1}^{Y_n}M_i$ is not bounded. Therefore, we write 
\begin{multline}\label{lowerbound1}
\Et\left(\sum_{i=1}^{Y_n}M_i-\Deltat(t)\right)=\Et\left(\left(\sum_{i=1}^{Y_n}M_i-n^3\right)\un_{\{\sum_{i=1}^{Y_n}M_i>n^3\}}\right)\\+\Et\left(\min\left(\sum_{i=1}^{Y_n}M_i,n^3\right)-\Deltat(t)\right).
\end{multline}

The first term on the right-hand side is negligible. Indeed, using Cauchy--Schwarz inequality and setting $\tilde{{\mathcal A}}_n(t)=\{M_i\le \ln^2(n), \forall 1\le i \le n^2\}$
\begin{align*}\Et\left(\left(\sum_{i=1}^{Y_n}M_i-n^3\right)\un_{\{\sum_{i=1}^{Y_n}M_i>n^3\}}\right)&\le\Et\left(\left(\sum_{i=1}^{Y_n}M_i\right)^2\right)^{1/2}\Pt\left(\sum_{i=1}^{Y_n}M_i>n^3\right)^{1/2}\\&=O\left(n^2\Pt(\tilde{{\mathcal A}}^c_n(t))^{1/2}\right)\end{align*}
and we will show in Appendix \ref{rareevent} that on the event ${\mathcal E}_n(t),$ $\Pt(\tilde{{\mathcal A}}^c_n(t))\leq n^{-20}$.

To bound the second term on the right-hand side of \eqref{lowerbound1}, we will apply again Lemma \ref{encadrement}. Using the event $\mathcal{A}_n(t)=\{\text{no component of }\cGt^{e}(t,h) \text{ has size greater than } \ln^2(n)\}$, we have
\begin{multline*}0\leq\Et\left(\min\left(\sum_{i=1}^{Y_n}M_i,n^3\right)-\Deltat(t)\right)\le 
\Et\left(\left(\sum_{i=1}^{Y_n}M_i-\Deltat(t)\right)\un_{\mathcal{A}_n(t)}\right)\\ +n^3\Pt\left(
\mathcal{A}^{c}_n(t)\right).
\end{multline*}
 On the event $\mc{E}_n(t)$, let us recall that $m^{+}_{n,1}$ stands for the conditional first moment of the total progeny of  the subcritical ${\BGW(n-L_n(t),\frac{t+h}{n})}$ given $\mc{F}_t$.  Denoting by $m^{-}_{n,1}$ the conditional expectation of the minimum of  the total progeny of a $\BGW(n-\Lt(t)-\ln^2(n),\frac{t+h}{n})$ and $\ln^2(n)$, we have on the event $\mc{E}_n(t)$,
\[
\Et\left[\left(\sum_{i=1}^{Y_n}M_i-\Deltat(t)\right)\un_{\mathcal{A}_n(t)}\mid Y_n\right] \le Y_n (m^{+}_{n,1}-m^{-}_{n,1})+\ln^2(n)\Et[|\mathcal{R}|\un_{\mathcal{A}_n(t)}\mid Y_n].\]
Let us recall that if $T$ is the total progeny of a $\BGW(n,p)$ for $p<\frac{1}{n}$, its first moment is $\frac{1}{1-np}$  and  there exists a constant $C$ such that 
$\E(T\un_{\{T> \ln(n)^2\}})\leq C \P(T> \ln^2(n))^{1/2}=O(n^{-9})$ since $T$ has finite exponential moments. In particular,  
\[\E(\min(T,\ln^2(n)))=\E(T)-\E(T\un_{\{T>\ln^2(n)\}})+\ln^2(n)\P(T>\ln^2(n))\geq \E(T)+O(n^{-9}).\]
 With the notation $f_2(t,x):=\frac{1}{1-t(1-x)}$, we obtain on the event $\mc{E}_n(t)$:
\[
m^{+}_{n,1}-m^{-}_{n,1}\leq f_2\left(t+h,\rho(t)+\frac{X_n(t)}{\sqrt{n}}\right)- f_2\left(t+h,\rho(t)+\frac{X_n(t)}{\sqrt{n}}+\frac{\ln^2(n)}{n}\right)+O(n^{-9}).\]
By property \ref{encadrement:A} of Lemma \ref{encadrement}
\[\Et[|\mathcal{R}|\un_{\mathcal{A}_n(t)}\mid Y_n]\le\frac{Y_n^2\ln^2(n)}{n\left(1-\left(\rho(t)+\frac{X_n(t)}{\sqrt{n}}\right)\right)}. \]
Therefore, on $\mathcal{E}_n(t)$, 
\begin{eqnarray}
\Et\Big(\Big(\sum_{i=1}^{Y_n}M_i&-&\Deltat(t)\Big)\un_{\mathcal{A}_n(t)}\mid Y_n\Big) \nonumber \\
&\leq & Y_n\left(f_2\left(t+h,\rho(t)+\frac{X_n(t)}{\sqrt{n}}\right)-f_2\left(t+h,\rho(t)+\frac{X_n(t)}{\sqrt{n}}+\frac{\ln^2(n)}{n}\right)\right)\nonumber\\
&&+Y_nO(n^{-9})+\frac{Y_n^2\ln^4(n)}{n\left(1-(\rho(t)+\frac{X_n(t)}{\sqrt{n}})\right)}  \nonumber\\
&\leq & C\frac{Y_n^2\ln^4(n)}{n}.
\end{eqnarray}
 
The last inequality comes from the fact that for $t>t_{0}$, the denominator of  $f_2(t,\rho(t))$ i.e. ${1-t(1-\rho(t))}$ is away from 0, as well as the quantity $1-t(1-\frac{L_n(t)}{n})$ on the event $\mathcal{E}_n(t)$, and a fortiori $1-t(1-\frac{L_n(t)}{n}-a)$ for all $a>0$. Taking the expectation in the last upper bound, we obtain on $\mathcal{E}_n(t)$,

\begin{eqnarray}\label{lowerbound1'}
\Et\left[\min\left(\sum_{i=1}^{Y_n}M_i,n^3\right)-\Deltat(t)\right]\le  C \ln^4( n)(nh^2+h)+n^3\Pt(
\mathcal{A}^{c}_n(t)),
\end{eqnarray}
and therefore
\begin{eqnarray}\label{mom13}
\Et\left[\sum_{i=1}^{Y_n}M_i-\Deltat(t)\right]\le   C(\ln^4(n)(nh^2+h)+n^{-8})+n^3\Pt(\mathcal{A}^{c}_n(t)).
\end{eqnarray}
Recalling \eqref{mom11}, we obtain, using Lemmas \ref{lem-behrisch} and \ref{boundan}
\begin{equation}\label{deltaX1}
\left|\Et(\Delta X_n(t)) - ha(t, X_n(t))\right|\leq r_{1,n}(h,X_n(t))
\end{equation}
with $r_{1,n}(h,X_n(t))\leq C\left(\ln^4 (n)(\sqrt{n}h^2+\frac{h}{\sqrt{n}})+ \frac{X^2_n(t) h}{\sqrt n}+n^{-7.5}\right)\un_{\mathcal{E}_n(t)}+\sqrt{n}\un_{\mathcal{E}^c_n(t)}$. 
\end{proof}

\paragraph*{Second moment.}
 We now turn to the estimation of $\Et(\Delta X_n(t)^2).$ We are going to prove
 \begin{lem}\label{lem:deltaX2}
For $x\in \RR$, $t>1$, $h>0$ and $n\in\NN^*$, set 
\begin{eqnarray*}
&&b(t)= \frac{{\rho(t)(1-\rho(t))}}{(1-t(1-\rho(t)))^{3}} \text{ and }
r_{2,n}(h,x) = \left(hn^{-0.3}+n^{-7.5}\right)\un_{\{|x|\leq  n^{0.2}\}}+n\un_{\{|x|> n^{0.2}\}}.
\end{eqnarray*} 
 
 There exists a constant $C>0$, such that for $n>t_1$, $h\leq n^{-1}$ and $t\in [t_0,t_1]$ 
\begin{equation}\label{deltaX2}
\left| \Et[(\Delta X_n(t))^2] - hb(t)\right| \leq C r_{2,n}(h,X_n(t)).
\end{equation}
\end{lem}
\begin{proof}
By definition, $n\Et[\Delta X_n(t)^2]$ is equal to 
\begin{multline*}
 \Et[\Delta L^{2}_n(t)] +n^2(\rho(t+h)-\rho(t))^2-2n\Delta L_n(t)(\rho(t+h)-\rho(t))]\\
=\Et[\Delta L^{2}_n(t)] -n^2(\rho(t+h)-\rho(t))^2-2n^{3/2}\left(\rho(t+h)-\rho(t)\right)\Et[\Delta X_n(t)].
\end{multline*}

As before we make use of Lemma \ref{lem-behrisch} to compute $\Et[\Delta L^{2}_n(t)]$: 
\[\Et[\Delta L^{2}_n(t)]=\Et[\Deltat^2(t)]+O(n^{-7}).\]

 From now on, the computations will be performed only on  the event $\mathcal{E}_n(t)$.   
Using Lemma \ref{encadrement}, the term $\Et[\Deltat^{2}(t)]$ is bounded from above by the second moment of the sum of $Y_n$ i.i.d. copies of the total progeny of a $\BGW(n-\Lt(t),\frac{t+h}{n})$. Hence, denoting by $v_n^+$ the variance of the total progeny of a $\BGW(n-\Lt(t),\frac{t+h}{n})$, classical variance computations give, 
\begin{eqnarray}
\label{eq-majcarre}
\Et[(\Deltat(t))^2]&\le& \Et\left[\left(\sum_{i=1}^{Y_n}M_i\right)^2\right]\\
&=& (\Et[{Y_n}]m_{n,1}^+)^2+ \var(Y_n\mid \mc{F}_t)(m_{n,1}^+)^2+\Et[Y_n]v_n^+.\nonumber
\end{eqnarray}
In the right-hand side of \eqref{eq-majcarre}, the second and third terms will give the main contribution and yield the diffusion term $nhb(t,X_t)$. Since the expectation and variance of the total progeny of a subcritical $\BGW(n,p)$ process are equal to $\frac{1}{1-np}$ and $\frac{np(1-p)}{(1-np)^3}$ respectively, we can write using the function $f_3(t,x)=\frac{x(1-x)}{(1-t(1-x))^3}$,
\begin{align}\label{uppervar}\var(Y_n&\mid  \mc{F}_t)(m_{n,1}^+)^2+\Et[Y_n]v_n^+\nonumber\\ =&\Lt(t)(n-\Lt(t))\frac{h}{n-t}\left(1-\frac{h}{n-t}\right)\frac{1}{\left(1-(n-\Lt(t))\frac{t+h}{n}\right)^2}\nonumber\\
&+\Lt(t)(n-\Lt(t))\frac{h}{n-t}\frac{(n-\Lt(t))\frac{t+h}{n}(1-\frac{t+h}{n})}{\left(1-(n-\Lt(t))\frac{t+h}{n}\right)^3}\nonumber\\
=&nhf_3\left(t+h,\frac{\Lt(t)}{n}\right)\left(1-\frac{h}{n-t}-\frac{t}{n}(1-\frac{h}{n-t})(1-\frac{\Lt(t)}{n})(t+h)\right)\nonumber\\
=&nh\left(f_3(t,\rho(t))+O\left(\frac{|X_n(t)|}{\sqrt{n}}\left|\frac{\partial f_3}{\partial x}(t,\rho(t))\right|+h\left|\frac{\partial f_3}{\partial t}(t,\rho(t))\right|\right)\right)\left(1+O(\frac{1}{n})\right)\nonumber\\
=&nh\left(f_3(t,\rho(t))+O\left(\frac{|X_n(t)|}{\sqrt{n}}+h+\frac{1}{n}\right)\right).
\end{align}

The other contributions in the estimation of the upper bound of $n\Et[\Delta X_n(t)^2]$ are the three following terms
  \[  (\Et[{Y_n}]m_{n,1}^+)^2-n^2(\rho(t+h)-\rho(t))^2-2n^{3/2}(\rho(t+h)-\rho(t))\Et[\Delta X_n(t)].\]
Since $\Et[{Y_n}]m_{n,1}^+=nh f_1(t+h, \frac{\Lt(t)}{n})(1+O(\frac{1}{n}))$ and $\rho'(t)=f_1(t,\rho(t))$, the  zero orders of the first two terms cancel:

\begin{equation} \label{uppervar2}
 (\Et[{Y_n}]m_{n,1}^+)^2-n^2(\rho(t+h)-\rho(t))^2=n^2h^2 O\left(h+\frac{1}{n}+\frac{|X_n(t)|}{\sqrt{n}}\right). 
\end{equation}
Therefore,  Lemma \ref{lem:deltaX}, giving the expression of $\Et[\Delta X_n(t)]$ together with \eqref{uppervar} and \eqref{uppervar2} yield the following upper bound for $\Et[\Delta X_n(t)^2]$:
\begin{eqnarray*}
\Et[\Delta X_n(t)^2]&\leq& hf_3(t,\rho(t))+C\left((h^2+nh^3)\ln^4(n)+\frac{h}{n}+n^{-8}\right)\\
&&+C\left(\left(\frac{h}{\sqrt{n}}+\sqrt{n}h^2\right)|X_n(t)|+h^2|X_n(t)|^2\right)\un_{\mc{E}_n(t)}+n\un_{\mathcal{E}^{c}_n(t)}.
\end{eqnarray*}
Since we based the upper bound of $\Et[\Delta X_n(t)^2]$ on the inequality \[\Et[\Deltat (t)^2]\leq \Et\left[\left(\sum_{i=1}^{Y_n}M_i\right)^2\right],\] we just have to find an upper bound of
$\Et\left[\left(\sum_{i=1}^{Y_n}M_i\right)^2-(\Deltat(t))^2\right]$
in order to get a lower bound of $\Et[\Delta X_n(t)^2]$. 
 We proceed as in \eqref{lowerbound1} with a slight modification. Introducing the event \[\mc{D}_n(t):=\left\{\sum_{i=1}^{Y_n}M_i>n^{\theta}\right\}\] for some $0<\theta<1$,
\begin{eqnarray}
\Et\left[\left(\sum_{i=1}^{Y_n}M_i\right)^2-\Deltat(t)^2\right]&=&\Et\left[\left(\left(\sum_{i=1}^{Y_n}M_i\right)^2-n^{2\theta}\right)\un_{\mc{D}_n(t)}\right]\nonumber\\
&&+\Et\left[\min\left(\sum_{i=1}^{Y_n}M_i,n^\theta\right)^2-\Deltat(t)^2\right].\nonumber
\end{eqnarray}

Using Cauchy--Schwarz inequality and the fact that $Y_n$ is bounded by $n^2$, the first term is smaller than $Cn^4\Pt(\mc{D}_n(t))^\frac{1}{2}$. 

When $h\leq \frac{1}{n}$, $\Pt(\mc{D}_n(t))\leq \Pt(\tilde{\mc{A}}_{n}^{c}(t))+\Pt(Y_n\geq n^{\theta/2})\leq Cn^{-20}$ . Indeed,   by Lemma \ref{boundan}, $\Pt(\tilde{\mc{A}}_{n}^{c}(t))=O(n^{-20})$. Moreover,  since the conditional distribution of $Y_n$ given $\mc{F}_t$ follows the binomial $\Bin(\Lt(t)(n-\Lt(t)),\frac{h}{n-t})$ distribution, it is stochastically dominated by the binomial $\Bin(\frac{n^2}{4},\frac{h}{n-t_1})$ distribution.  Therefore, if $h\leq \frac{1}{n}$, $\Pt(Y_n\geq n^{\theta/2})=O(n^{-20})$ by application of the exponential Markov's inequality.

 To  bound from above the second term we note that,
\begin{align*}
\min&\left(\sum_{i=1}^{Y_n}M_i, n^\theta\right)^2-(\Deltat(t))^2\\
&\le\left(\min\left(\sum_{i=1}^{Y_n}M_i,n^\theta\right)^2-(\Deltat(t))^2\right)\un_{\Deltat(t)<n^\theta}\\
&\le\left(\min\left(\sum_{i=1}^{Y_n}M_i, n^\theta\right)+\Deltat(t)\right)\left(\min\left(\sum_{i=1}^{Y_n}M_i,n^\theta\right)-\Deltat(t)\right)\un_{\Deltat(t)<n^\theta}\\
&\leq2n^\theta\left(\min\left(\sum_{i=1}^{Y_n}M_i,n^\theta\right)-\Deltat(t)\right)\un_{\Deltat(t)<n^\theta}\\
&\leq2n^\theta\left(\sum_{i=1}^{Y_n}M_i-\Deltat(t)\right).
\end{align*}

Using equation \eqref{mom13} to bound 
$
\Et\left[\sum_{i=1}^{Y_n}M_i-\Deltat(t)\right]$, 
we obtain that if $h\leq 1/n$, 
\[\Et\left[\left(\sum_{i=1}^{Y_n}M_i\right)^2-\Deltat(t)^2\right]
\le   C(n^{\theta}\ln^4(n)h+n^{-7+\theta}).\] 

Since $b(t)=f_3(t,\rho(t))=\dfrac{\rho(t)(1-\rho(t))}{(1-t(1-\rho(t)))^3}$,  gathering all the above estimations we obtain for $h \leq \frac1n$, 
\[
\left|\Et(\Delta X_n(t)^2) - hb(t)\right|\leq C\left(\frac{|x|}{\sqrt{n}}h+|x|^2h^2 +n^{\theta-1}\ln^{4}(n)h+n^{-8+\theta}\right)\un_{|x|\leq n^{0.2}}+n\un_{|x|>n^{0.2}}
\]
Taking $\theta=\frac12$ yields for $h\leq \frac{1}{n}$, \[\left|\Et(\Delta X_n(t)^2) - hb(t)\right|\leq Cr_{2,n}(h,X_n(t))\] with  
$r_{2,n}(h,x)=(n^{-0.3}h+n^{-7.5})\un_{|x|\leq n^{0.2}}+n\un_{|x|>n^{0.2}}$.

\end{proof}
\paragraph*{Third moment.} 
For the third moment, we only need to obtain an upper bound on
$\Et[|\Delta X_n(t)|^3]$. 

\begin{lem}\label{lem:deltaX3}
There exists a constant $C>0$, such that for every integer $n>t_1$, reals $t\in [t_0,t_1]$ and $0<h\leq n^{-1}$,  
\begin{equation}\label{deltaX3}
\Et(|\Delta X_n(t)|^3) \leq C r_{3,n}(h,X_n(t))
\end{equation}
where  
\begin{eqnarray*}
r_{3,n}(h,x)&= & hn^{-\frac12}(\ln n)^6+n^{-7.5}+n^\frac{3}{2}\un_{\{|x|> n^{0.2}\}}.
\end{eqnarray*}

\end{lem}
\begin{proof}
As before, we first restrict the study to the event $\mathcal{E}_n(t)$. On this event, 
\begin{alignat*}{2}
    \Et[|\Delta X_n(t)|^3] = \Et&\left[\mid  n^{-\frac32}\Delta L_n(t)^3-3n^{-\frac12}\Delta L_n(t)^2(\rho(t+h)-\rho(t))\right.\\ &\left.+3n^{\frac12}\Delta L_n(t)(\rho(t+h)-\rho(t))^2 -n^{\frac32}(\rho(t+h)-\rho(t))^3\mid\right].
\end{alignat*}
Using again Lemma \ref{lem-behrisch}, and using the fact that on $\mathcal{A}_{n}(t)$, $|\Deltat(t)|\leq(\ln n)^2 Y_n,$ we have on $\mathcal{E}_n(t)$:
\begin{eqnarray*}
   \Et[|\Delta X_n(t)|^3] &\leq & n^{-\frac32}\ln(n)^6 \Et[Y^{3}_n]+3n^{-\frac12}\ln(n)^4\Et[Y^{2}_n]|\rho(t+h)-\rho(t)|\\
    && +3n^{\frac12}\ln(n)^2\Et[Y_n](\rho(t+h)-\rho(t))^2+n^{\frac32}|\rho(t+h)-\rho(t)|^3\\
    && +n^{\frac32}\Pt(\mathcal{A}_{n}^c(t))(1+O(h))+O(n^{-7.5}).
\end{eqnarray*}

As the conditional distribution of $Y_n$ is stochastically dominated by $\Bin(\frac{n^2}{4},\frac{h}{n-t_1})$, there exists a constant $C$ such that for  $t_0\leq t\leq t_1$ and $h\leq \frac{1}{n}$,  $\Et[Y^{k}_n]\leq C n h$ for $k\in\{1,2,3\}$. Therefore, there exists a constant $C$ such that for  $t_0\leq t\leq t_1$ and $h\leq \frac{1}{n}$, 
\[\Et[|\Delta X_n(t)|^3]\leq C (hn^{-\frac12}\ln^{6}(n)+n^{-7.5})+n^{\frac32}\un_{\mathcal{E}^{c}_n(t)}.\] 
\end{proof}

\subsection{Assumptions of Proposition \ref{normanth1} for the sequence \texorpdfstring{$Z_n$}{Zn}}\label{sub-moments-aux}

This section is devoted to checking that the sequence of processes $Z_n:=(uX_n)\circ v^{(-1)}$, satisfies the assumptions of Proposition \ref{normanth1} where the process $Z$ is nothing but the standard Brownian motion. 
For this purpose we want to check these assumptions with $\alpha\equiv 0$, $\beta\equiv 1$, and take the origin of time at $v(t_0)$. 
In this case, the remainder functions are $\epsilon_1\equiv 0$, $\epsilon_2\equiv 0$ and $\epsilon_3\equiv \frac{4}{\sqrt {2\pi}} h^{3/2}$.

In this setting, checking Assumption \hyperlink{hypA234}{A2} is straightforward, and the sum in  Assumption \hyperlink{hypA234}{A3} converges to 0 since it is of order $N_n^{-1/2}$. The convergence of $X_n(t_0)$ to a centered Gaussian variable with variance $\sigma^{2}(t_0)=\frac{\rho(t_0)(1-\rho(t_0))}{(1-t_0(1-\rho(t_0)))^2}=\frac{v(t_0)}{u(t_0)^2}$ is a consequence of the classical central limit Theorem for the order of the giant component (see \cite{BBDLV,Pittel,Stepanov}). As $v$ is an increasing smooth function from $(1,+\infty)$ to $(0,+\infty)$ and $u$ is a smooth function on $(1,+\infty)$ that is never equal to zero, this implies that $Z_n(v(t_0))$ converges to a centrered Gaussian random variable with variance $v(t_0)$, which implies Assumption \hyperlink{hypA234}{A4}. 

Let us now prove Assumption \hyperlink{hypA1}{A1}, concerning the first three moments of the increment of $Z_n$.
The estimations of the increment of $Z_n$ are based on the identity:
\begin{eqnarray}\label{deltaZ}
\Delta Z_n(t)&:= &Z_n(t+h)-Z_n(t)\nonumber\\
&=& u\circ v^{(-1)}(t+h) \left(X_n\circ v^{(-1)}(t+h)-X_n\circ v^{(-1)}(t)\right)\nonumber\\
&&+X_n\circ v^{(-1)}(t)\left(u\circ v^{(-1)}(t+h)-u\circ v^{(-1)}(t)\right)\nonumber\\
&=& (u\circ v^{(-1)}(t)+O(h)) \left(X_n\circ v^{(-1)}(t+h)-X_n\circ v^{(-1)}(t)\right)\nonumber\\
&&+X_n\circ v^{(-1)}(t)\left(h(u\circ v^{(-1)})'(t)+O(h^2)\right)
\end{eqnarray}
\paragraph*{The first moment of $\Delta Z_n$.} 
By Lemma \ref{lem:deltaX} and the smoothness of $v^{(-1)}$, there exists a positive constant  $C$ such that for every $t\in [v(t_0),v(t_1)]$, $n>t_1$  and $h\leq \frac{C}{n}$, 
\begin{multline*}
\Ev\left[X_n\circ v^{(-1)}(t+h)-X_n\circ v^{(-1)}(t)\right]\\=\left(\frac{h}{v'\circ v^{(-1)}(t)}+O(h^2)\right)a(v^{(-1)}(t),X_n\circ v^{(-1)}(t))\\
+O\left(hn^{-0.1}+n^{-7.5}+\sqrt{n}\un_{|X_n\circ v^{(-1)}(t)|>n^{-0.2}}\right).
\end{multline*}
Noting that $a(t,x)=-\frac{u'(t)}{u(t)}x$, we have:
\[(u\circ v^{(-1)})(t) \frac{a(v^{(-1)}(t),X_n\circ v^{(-1)}(t))}{v'\circ v^{(-1)}(t)}+X_n\circ v^{(-1)}(t)(u\circ v^{(-1)})'(t)=0.\]
Hence, the expectation of the zero order term of \eqref{deltaZ} cancels. 
We obtain that if $n>t_1$, $h\leq \frac{C}{n}$, and $t\in [v(t_0),v(t_1)]$ then
\begin{eqnarray}\label{deltaZ1}
\Ev\left[\Delta Z_n(t)\right]
&=& O\left(h^2 |X_n\circ v^{(-1)}(t)|+hn^{-0.1}+n^{-7.5}+\sqrt{n}\un_{|X_n\circ v^{(-1)}(t)|>n^{-0.2}}\right).\nonumber\\
&=&O(hn^{-0.1} +n^{-7.5}+\sqrt{n}\un_{|X_n\circ v^{(-1)}(t)|>n^{0.2}}).
\end{eqnarray}
\paragraph*{The second moment of $\Delta Z_n$.} 
Let $t\in[v(t_0),v(t_1)]$.
 By formula \eqref{deltaZ},
\begin{eqnarray}
\Ev\left[(\Delta Z_n(t))^2\right]
&=& (u\circ v^{(-1)}(t)+O(h))^2\Ev\left[(\Delta (X_n\circ v^{(-1)})(t))^2\right] \nonumber\\
&&+(X_n\circ v^{(-1)}(t))^2\left(h(u\circ v^{(-1)})'(t)+O(h^2)\right)^2\nonumber\\
&&+O(h)X_n\circ v^{(-1)}(t)\Ev\left[\Delta (X_n\circ v^{(-1)})(t)\right].\nonumber
\end{eqnarray}
 By  Lemma \ref{lem:deltaX}, there exists a positive constant $C$ such that if $h\leq \frac{C}{n}$ then, on the event $\mathcal{E}_n(v^{-1}(t))=\{|X_n\circ v^{(-1)}(t)|\leq n^{0.2}\}$, 
\[\left|\Ev\left[\Delta (X_n\circ v^{(-1)})(t)\right]\right|=O\left(h|X_n\circ v^{(-1)}(t)|+hn^{-0.1}+n^{-7.5}\right).\]
 Hence taking $h\leq \frac{C}{n}$,  the last two terms on the right-hand side are, on the event $\mathcal{E}_n(v^{(-1)}(t))$, 
\[O\left(h^2(X_n\circ v^{(-1)}(t))^2+h|X_n\circ v^{(-1)}(t)|(hn^{-0.1}+n^{-7.5})\right)=O(h^2n^{0.4}+hn^{-7.3}).\]
Moreover, by Lemma \ref{lem:deltaX2}, for $h\leq \frac{C}{n}$,  on the event $\mathcal{E}_n(v^{(-1)}(t))$,
\begin{eqnarray*}
\Ev\left[(\Delta (X_n\circ v^{(-1)})(t))^2\right]&=&
\left(\frac{h}{v'\circ v^{(-1)}(t)}+O(h^2)\right)b(v^{(-1)}(t))\nonumber\\
&&+O\left(hn^{-0.3}+n^{-7.5}\right).
\end{eqnarray*}
Noticing that $b=v'/u^2$, we obtain that, on the event $\mathcal{E}_n(v^{(-1)}(t))$,
\[(u\circ v^{(-1)}(t)+O(h))^2\Ev\left[(\Delta (X_n\circ v^{(-1)})(t))^2\right]=h+O\left(h^2+hn^{-0.3}+n^{-7.5}\right).\]
Therefore, for $h\leq \frac{C}{n}$, 
\begin{equation}\label{deltaZ2}
\Ev\left[(\Delta Z_n(t))^2\right]=h+O(hn^{-0.3}+n^{-7.5}+n\un_{|X_n\circ v^{(-1)}(t)|>n^{0.2}}).
\end{equation}

\paragraph*{The third moment of $\Delta Z_n$.} 
By equation \eqref{deltaZ}, 
\[
|\Delta Z_n(t))|=O\left(|\Delta X_n\circ v^{(-1)}(t)|+h|X_n\circ v^{(-1)}(t)|\right).
\]
We bound separately the  conditional expectation of the four terms obtained in the expansion of the cube on the right-hand side using Lemmas  \ref{lem:deltaX2} and \ref{lem:deltaX3} together with the simple estimation 
$\Ev\left[|\Delta X_n\circ v^{(-1)}(t)|\right]=O(\sqrt{n})$: there exists a positive constant $C$ such that on the event $\mathcal{E}_n(v^{(-1)}(t))$, if $h\leq \frac{C}{n}$ then 
\begin{eqnarray*}
\Ev\left[|\Delta Z_n(t)|^3\right]&=& O(hn^{-0.4}+n^{-7.5})+O(h|X_n\circ v^{(-1)}(t)|(h+n^{-7.5}))\\ 
&&+O(\sqrt{n}h^2(X_n\circ v^{(-1)}(t))^2)+O(h^3|X_n\circ v^{(-1)}(t)|^3)\\
&=&O(hn^{-0.1}+n^{-7.5}). 
\end{eqnarray*}
Therefore, for $h\leq \frac{C}{n}$, 
\begin{equation}\label{deltaZ3}
\Ev\left[|\Delta Z_n(t)|^3\right]=O(hn^{-0.1}+n^{-7.5}+n^{3/2}\un_{|X_n\circ v^{(-1)}(t)|>n^{0.2}}). 
\end{equation}

\paragraph*{Assumption A1.} 
We now use  the previous estimates \eqref{deltaZ1}, \eqref{deltaZ2} and \eqref{deltaZ3} to check Assumption \hyperlink{hypA1}{A1} of Proposition~\ref{normanth1}.
We have to show that there exists a sequence of positive integers $(N_n)_n$ that converge to $+\infty$ such that, for any $0<s<t$~,
\[\sum_{k=0}^{N_n-1}\sum_{i=1}^{3}\E\left(\left|e_{i,n}\left(s+k\frac{t-s}{N_n},\frac{t-s}{N_n},Z_{n}(v(t_0)+s+k\frac{t-s}{N_n})\right)\right|\right)\underset{n\rightarrow +\infty}{\rightarrow} 0,\]
where, for some positive constants $C_1$, $C_2$ depending only on $t_0$ and $t$, whenever $n>C_1$,  $n^{-7}\leq h\leq  n^{-1.1}$,   $0\leq\tau\leq t$ and $i\in\{1,2,3\}$
\[e_{i,n}\left(\tau,h,z\right)=C_2(hn^{-0.1}+n^{i/2}\un_{|z|>(u\circ v^{(-1)}(v(t_0)+\tau)n^{0.2}}).\]

Fixing $N_n=n^5$,  Assumption \hyperlink{hypA1}{A1} is true if we show that \[3n^{3/2}\sum_{k=0}^{N_n-1} \P\left(\left|X_n\circ v^{(-1)}\left(v(t_0)+s+k\frac{t-s}{N_n}\right)\right|>n^{0.2}\right)\underset{n\rightarrow +\infty}{\rightarrow} 0.\] This is a consequence of Lemma \ref{bounden}.

\section{Tightness\label{Tightness}}

Our approach to prove the tightness of $Z_n$ relies on the criterion by Billingsley (\cite{Billingsley}, Theorem 13.5).
Namely we have to prove that, there exist $\alpha>1$ and a constant $C>0$ such that for $v(t_0)\leq r<s<t\leq v(t_1)$  
\begin{equation}\label{tightness1}
\E\left[(Z_n(s)-Z_n(r))^2(Z_n(t)-Z_n(s))^2\right]\le C(t-r)^{\alpha}.
\end{equation}

We will treat separately the case where $t-r$ is very small comparatively to $n$ and when it is not. More precisely we will distinguish between  $t-r\le n^{-5}$ and $t-r> n^{-5}$.

\subsection*{The case $t-r \le n^{-5}$}
As $Z_n(t)=\Xmu(v^{(-1)}(t)),$ and $v$ and $v^{(-1)}$ are Lipschitz on the intervals we are considering, to prove \eqref{tightness1} for $t-r\le n^{-5}$ it suffices to prove for some constant $D$ that, for every $t_0\leq r<s<t\leq t_1$ such that $t-r\le Dn^{-5}$,
\begin{equation}
\E\left[\left(\Xmu(s)-\Xmu(r)\right)^2\left(\Xmu(t)-\Xmu(s)\right)^2\right]\le C(t-r)^{\alpha}.
\end{equation}
Since $\Xmu(s)-\Xmu(r)$ is equal to 
\[\frac{(u(s)-u(r))L_n(s)+u(r)(L_n(s)-L_n(r))-(u(s)\rho(s)-u(r)\rho(r))n}{\sqrt{n}},
\]
we have by straightforward calculation and by bounding $L_n$ by $n$: 
\begin{equation}\label{tightness2}
(\Xmu(s)-\Xmu(r))^2\le O(n(s-r)^2)+C\frac{(L_n(s)-L_n(r))^2}{n}.
\end{equation}

The same is evidently true for $(\Xmu(t)-\Xmu(s))^2$. Observing that for $t-r\le n^{-5}$, 
\begin{equation}\label{tightness3}
n(s-r)^2 \le n(t-r)^{0.9}(t-r)^{1.1}\le\frac{(t-r)^{1.1}}{n^{3.5}},
\end{equation} and bounding $L_n$ by $n$,  we are reduced to show that for some $\alpha>1$, 
\[\E\left[\frac{(L_n(s)-L_n(r))^2(L_n(t)-L_n(s))^2}{n^2}\right]\le C (t-r)^\alpha.\]
As  the difference $t-r$ is assumed very small, the term in the expectation will be zero with a very high probability. Indeed, for it to be non-zero it is necessary that one edge is drawn during each of the intervals $[r,s]$ and $[s,t]$. Brutally bounding $L_n(t)-L_n(r)$ by $n$ and noting that the number of edges drawn in two distinct intervals are negatively correlated, we have
\begin{multline}\label{tightness4}
\E\left[\frac{(L_n(s)-L_n(r))^2(L_n(t)-L_n(s))^2}{n^2}\right]\le n^2\P\left(\Bin(n^2,(t-r)/n)>0\right)^2\\=O(n^4(t-r)^2)=O(1)(t-r)^{1.1}.
\end{multline}

Hence \eqref{tightness1} holds.

\subsection*{The case $t-r > n^{-5}$}
In this case our strategy is to subdivide the interval $[r,t]$ so that we can apply our results from Section \ref{convergencefdd}. Thus, we take as increment $h=n^{-5}$.
Set $k=\lfloor\frac{t-s}{h}\rfloor$,  $s_i=s+ih$ for $0\leq i\leq k$ and  $s_{k+1}=t$. We construct similarly $r_i=r+ih$ for $0\leq i\leq \lfloor\frac{r-s}{h}\rfloor$.

We have
\begin{equation*}
\E[(Z_n(s)-Z_n(r))^2(Z_n(t)-Z_n(s))^2]=\E\left[(Z_n(s)-Z_n(r))^2 \Ev[s][(Z_n(t)-Z_n(s))^2]\right].
\end{equation*}

We first determine an upper bound for 
$\Ev[s][(Z_n(t)-Z_n(s))^2]$.

We write, for any function $g$, $\Delta_i g:=g(s_{i+1})-g(s_{i}).$
 The decomposition 
\[
Z_n(t)-Z_n(s)=\sum_{i=0}^{k}\left(\Delta_i Z_n-\Ev[s_i][\Delta_i Z_n]\right)+\sum_{i=0}^{k}\Ev[s_i][\Delta_i Z_n]
\]
implies
\[(Z_n(t)-Z_n(s))^2\leq 2 \left(\sum_{i=0}^{k}\left(\Delta_i Z_n-\Ev[s_i][\Delta_i Z_n]\right)\right)^2+ 2\left(\sum_{i=0}^{k}\Ev[s_i][\Delta_i Z_n]\right)^2.
\]

Thus, by orthogonality of the centered terms: 
\[
\Ev[s][(Z_n(t)-Z_n(s))^2]\leq 2(S_1+S_2).
\]
with 
\[S_1=\sum_{i=0}^{k}\Ev[s]\left[\left(\Delta_i Z_n-\Ev[s_i][\Delta_i Z_n]\right)^2\right]\]
and \[S_2=\Ev[s]\left[\left(\sum_{i=0}^{k}\Ev[s_i][\Delta_i Z_n]\right)^2\right].\]

To find an upper bound of $S_2$ we use \eqref{deltaZ1} and get
\begin{equation*}
\Ev[s_i][\Delta_i Z_n] = O(hn^{-0.1}+n^{-7.5}+\sqrt{n}\un_{|X_n\circ v^{(-1)}(s_i)|>n^{0.2}}).
\end{equation*}
Note that for $i=k,$ this estimate remains true when replacing $h$ by $s_{k+1}-s_k<h$.
Hence, recalling $h=n^{-5}$
\[\left(\sum_{i=0}^{k}\Ev[s_i](\Delta_i Z_n)\right)^2=O\left((t-r)^2+n\left(\sum_{i=0}^k \un_{\mc{E}_n^c(v^{(-1)}(s_i))}\right)^2\right),\]
whence \[S_2=O\left((t-r)^2+n(k+1)^2\max_{0<i<k}\Pv[s](\mc{E}_n^c(v^{(-1)}(s_i)))\right).\]

It remains to deal with $S_1$. Using \eqref{deltaZ2}, the conditional second moment of $\Delta Z_n(s_i)$ is bounded by: 
\[
\Ev[s_i][(\Delta_i Z_n)^2]=  O(h+n^{-7.5}+n\un_{|X_n\circ v^{(-1)}(s_i)|>n^{0.2}}).
\]
Hence,
\[S_1 \leq 
\sum_{i=0}^{k}\Ev[s][(\Delta_i Z_n)^2]=O\left(t-r+n(k+1)\max_{0<i<k}\Pv[s](\mc{E}_n^c(v^{(-1)}(s_i)))\right).
\]

Putting together the estimates for $S_1$ and $S_2$ we have 
\[
\Ev[s][(Z_n(t)-Z_n(s))^2]=O\left(t-r+n(k+1)^2\max_{0<i<k}\Pv[s](\mc{E}_n^c(s_i))\right).
\]
Using this bound we obtain
\begin{multline*}
\E\left[(Z_n(s)-Z_n(r))^2(Z_n(t)-Z_n(s))^2\right] \le C(t-r)\E\left[(Z_n(s)-Z_n(r))^2\right]\\+C\E\left[(Z_n(s)-Z_n(r))^2 n(k+1)^2\max_{0<i<k}\Pv[s](\mc{E}_n^c(v^{(-1)}(s_i)))\right].
\end{multline*}

The second term is smaller than $n^{-10}$ by Lemma \ref{bounden}. The first term is bounded similarly as $\Ev[s][(Z_n(t)-Z_n(s))^2]$. Therefore, 
\[\E\left[(Z_n(s)-Z_n(r))^2(Z_n(t)-Z_n(s))^2\right] \le C(t-r)^2,\]
which finishes the proof of \eqref{tightness1} for $t-r>n^{-5}$.
\appendix
\section{Rare events\label{rareevent}}
In this Appendix we give an upper bound, uniform in $t\in[t_0,t_1]$, on the probability of the complement of the following events introduced in the article:
\begin{itemize}
\item $\mathcal{E}_n(t)$,  the event that, at time $t$, $|L_n(t)-n\rho(t)|\leq n^{0.7}$ or equivalently $|X_n(t)|\leq n^{0.2}$;
\item $\mathcal{A}_n(t)$,  the event that no component of $\cGt^{e}(t,h)$ has size greater than $\ln^2(n)$;
\item $\tilde{{\mathcal A}}_n(t)$, the event $\{M_i\le \ln^2(n),\quad \forall 1\le i \le n^2\}$, where the random variables $M_i$ are independent and distributed, conditional on $\mc{F}_n(t)$, as the total progeny of a $\BGW({n-\Lt(t)},\frac{t+h}{n})$.
\end{itemize}
We are going to show the following two lemmas:
\begin{lem}\label{bounden}
Let $0<\gamma<0.5$.  There exist positive constants $C$ and $C'$ such that for every $t\in[t_0, t_1]$,  \[\P(|X_n(t)|> n^{\gamma})\leq C'\exp(-Cn^{2\gamma}).\]
\end{lem}
\begin{lem}\label{boundan}
Let $0<\gamma<0.5$. There exists $h_0>0$ such that, for $0\leq h\le h_0,$ uniformly in $t\in[t_0, t_1]$, on the event $\{|X_n(t)|< n^{\gamma}\}$, for all $R>0$,
\[ \P(\mathcal{A}^{c}_n(t)\mid \mc{F}_n(t))=O(n^{-R}) \text{ and }\P(\tilde{\mathcal{A}}^{c}_n(t)\mid \mc{F}_n(t))=O(n^{-R}).\]
\end{lem}
\subsection*{Proof of Lemma \ref{bounden}}
Our main tool to prove Lemma \ref{bounden} is the following Theorem by Stepanov:

\begin{tref}[Theorem 1 in \cite{Stepanov}]
  Set $y_0>0$ and let $P_n(s)$ denote the probability that a $\ER(n, 1-e^{-s})$ graph is connected. 
  \begin{equation}
  \label{stepanov}
   P_n\left(\frac{y}{n}\right)=\left(1-\frac{y}{e^{y}-1}\right)(1-e^{-y})^n(1+o(1))   
  \end{equation}
  where the $o(1)$ is uniform in $y\ge y_0$.
\end{tref}

Note that Stepanov's result uses a different time parametrization (namely ${\frac{t}{n}=1-e^{-\frac{y}{n}}}$). For the sake of readability we will use Stepanov's parametrization, and prove the following lemma, which is equivalent to  Lemma \ref{bounden}.
\begin{lem}
\label{enbis}
Set  $\hat{L}_n(y)=L_n(n(1-e^{-y/n}))$ the largest order of connected components of a ${\ER(n,1-e^{-\frac{y}{n}})}$ graph.\\
For $1<y_0<y_1$ and  $0<\gamma<0.5$,  there exist positive constants $C$ and $C'$  such that for every $y\in[y_0,y_1]$ and $n\in \mathbb{N}$, 
\[\P\left(|\hat{L}_n(y)-n\rho(y)|> n^{0.5+\gamma}\right)\leq C'e^{-Cn^{2\gamma}}.\]
\end{lem}

\begin{proof}
The ideas come from the proof of Theorem 2 in Stepanov's article \cite{Stepanov}, however we will use them towards a different objective.
We call $E_{n,k}(s)$ the expected number of components with $k$ vertices in an $\ER(n, 1-e^{-s})$ graph.
As pointed out in Stepanov's article (Eq. (27) in \cite{Stepanov}),
\[E_{n,k}\left(\frac{y}{n}\right)=\binom{n}{k}e^{-yk(1-k/n)}P_k\left(\frac{y}{n}\right).\]

Our idea is to use the upper bound $\P(\hat{L}_n(y)=k)\le E_{n,k}(y/n)$. Obviously, this upper bound is extremely bad if $k$ is much smaller than $n$. Therefore, we will treat these values separately. Set $0<\epsilon<\min(\rho(y_0),1-\rho(y_1)).$ 
We now decompose 
\begin{multline}\label{decsum}
    \P(|\hat{L}_n(y)-n\rho(y)|> n^{\gamma+0.5})\le \P(\hat{L}_n(y)\le \epsilon n)+\sum_{\substack{\epsilon n\le k\le (1-\epsilon)n : \\ |\frac kn - \rho(y)|> n^{\gamma-0.5}}}E_{n,k}\left(\frac{y}{n}\right)\\+\P(\hat{L}_n(y)\ge (1-\epsilon) n).\end{multline}
Note that the first and third terms are monotonous in $y$. Therefore, we only have to bound them for $y=y_0$ and $y=y_1$, respectively. This can be done using a large deviation principle for $\hat{L}_n(y)$ at a fixed value of $y$ (see for example Theorem 3.1 of \cite{Oconnell}). It yields bounds of order $\exp(-Cn)$.

To deal with the second term we use Stirling's formula to get, for $\epsilon n\leq k\leq (1-\epsilon) n$,
\[E_{n,k}\left(\frac{y}{n}\right)=\frac{1}{\sqrt{2\pi k(1-\frac{k}{n})}\left(\frac{k}{n}\right)^k (1-\frac{k}{n})^{n-k}} e^{-yk(1-\frac{k}{n})}P_k\left(\frac{y}{n}\right)(1+o(1)).\]

Using  \eqref{stepanov}, there exists a constant $C_{y_0}>0$ only dependent on $y_0$ such that for $y\ge y_0$, \[P_n\left(\frac{y}{n}\right)\le C_{y_0}\left(1-\frac{y}{e^{y}-1}\right)(1-e^{-y})^n.\]
In order to bound from above $E_{n,k}\left(\frac{y}{n}\right)$, we have to deal with $P_k(\frac yn)=P_k(\frac kn\frac yk)$ where ${\frac kn y\geq \epsilon y_0}$, hence
\[E_{n,k}\left(\frac{y}{n}\right)\le C_{\epsilon y_0}\frac{1}{\sqrt{2\pi k(1-\frac{k}{n})}}\left(1-\frac{\frac{k}{n} y}{e^{\frac{k}{n} y}-1}\right)e^{n\phi(\frac{k}{n},y)},\]
with $\phi(x, y)=-xy(1-x)+x\ln(1-e^{-x y})-x\ln x-(1-x)\ln(1-x).$

This function cancels if and only if $x=0$ and $x=\rho(y)$. Indeed, if we express it as a function of $x$ and $\delta(x,y)=1-e^{-xy}-x$, it writes:
\[\phi(x, y)=x\ln(1+\frac{\delta}x)+(1-x)\ln(1-\frac{\delta}{1-x})=:\psi(x,\delta),\]
and $\dfrac{\partial \psi}{\partial\delta}(x,\delta)=-\dfrac{\delta}{(x+\delta)(1-x-\delta)}.$

We observe that $|\dfrac{\partial \psi}{\partial\delta}(x,\delta)|\geq 4|\delta|$, hence $\psi(x,\delta)\leq -2\delta^2$.
Therefore, $\psi$ vanishes only for $\delta=0$ (it increases before and decreases after 0 since ${-x\leq\delta\leq 1-x}$).

Now note that $\delta(\cdot,y)$ is a concave function such that $\delta(0,y)=\delta(\rho(y),y)=0.$ Furthermore, 
\begin{eqnarray*}&&\frac{\partial \delta(x,y)}{\partial x}|_{x=0}=y-1\ge y_0-1,\\
&&\frac{\partial \delta(x,y)}{\partial x}|_{x=\rho(y)}=y(1-\rho(y))-1=\lambda(y)-1<\lambda(y_0)-1<0.
\end{eqnarray*}
Consequently, for $k$ such that $\epsilon n\le k\le (1-\epsilon)n$ and $|\frac kn - \rho(y)|> n^{\gamma-0.5}$, we get ${\left|\delta\left(\frac{k}{n},y\right)\right|>Cn^{\gamma-0.5}}$, which implies
 $\left|\phi\left(\frac{k}{n},y\right)\right|>Cn^{2\gamma-1}$ and in turn,
\[E_{n,k}(\frac{y}{n})\le C'e^{-Cn^{2\gamma}}.\]

Summing over those values of $k$ shows that the  sum on the right-hand side of
\eqref{decsum} is smaller than $C'e^{-Cn^{2\gamma}}$.

\end{proof}
\subsection*{Proof of  Lemma \ref{boundan}}
To prove Lemma \ref{boundan}, let us start by showing that the event \[\mc{A}^{c}_n(t)=\{\cGt^{e}(t,h) \text{ has components of order greater than } \ln^{2}(n)\}\] is a rare event uniformly for $t\in [t_0,t_1]$ as $n$ tends to $+\infty$. 
The idea is that  on the event $\{|X_n(t)|< n^{\gamma}\}$, $\cGt^{e}(t,h)$ is distributed as a subcritical ER graph for $n$ large enough so we can apply the following result on subcritical ER graph:
\begin{tref}[Theorem 4.4 in \cite{vdh}]
Set $I(c)=c-1-\ln(c)$ and  $|\mc{C}_{max}|$ the order of the largest connected component of an $\ER(n,\frac{c}{n})$ graph for $c>0$. \\
For every $0<c<1$ and $a>0$, 
$\P\left(|\mc{C}_{\max}|\geq a\ln(n)\right)\leq  n^{1-aI(c)}$. 
\end{tref}
In our case, the parameter $c$ depends on $X_n(t)$: the graph $\cGt^{e}(t,h)$ is distributed as a  ${\ER(n-L_n(t),\frac{c_n(t)}{n-L_n(t)})}$ graph with $c_n(t)=\lambda(t)-\frac{X_n(t)}{\sqrt{n}}(t+h)+h(1-\rho(t))$.  
Since $I(c)$ is decreasing on $[0,1],$ the result follows as soon as we show that there exists $\epsilon\in]0,1[$ and $n_0$ not depending on $t$ such that, for $n\ge n_0,$ $c_n(t)<1-\epsilon$ on the event $\{|X_n(t)|\leq n^{\gamma}\}$. This holds for $h$ small enough
since, on the event $\{|X_n(t)|\leq n^{\gamma}\}$, $\frac{|X_n(t)|}{\sqrt{n}}(t+h)\leq 2t_1n^{\gamma-0.5}$ and $\sup_{t\in[t_0,t_1]}\lambda(t)<1$. 

To deal with the event $\tilde{{\mathcal A}}_n(t)=\{M_i\le (\ln(n))^2, \forall 1\le i \le n^2\}$, we observe  that, on the event $\{|X_n(t)|\leq n^{\gamma}\}$, the random variables  $M_i$, $1\le i\le n^2$, can be dominated by the total progeny of a subcritical Poisson $\BGW$ process with parameter $1-\epsilon$ for every $t\in [t_0, t_1].$ This follows a Borel distribution, which admits an exponential moment.

\section{Proof of Proposition \ref{normanth1}\label{annex:proofthnornm}}
Let $k\in \mathbb{N}^*$.  Consider $k$ positive reals $0<s_1<\ldots<s_k$  and $k$ functions $f_1,\ldots f_k$ in $\mc{H}$. We shall prove that  $\E\left(\prod_{i=1}^{k}f_i(Z_n(s_i))\right)$  converges towards $\E\left(\prod_{i=1}^{k}f_i(Z(s_i))\right)$ as $n$ tends to $+\infty$. 
To study the difference of these two expectations, set 
   \[I_{n,u,v}(f)=\E\left(f(Z_{n}(v))\mid \mathcal{F}_{n}(u))-T_{v-u}f(Z_{n}(u)\right)\quad \forall 0\leq u<v \text{ and } \forall f\in \mc{H}. 
   \] 
With this notation, 
\[\E\left(f_1(Z_n(s_1))\right)=\E\left(I_{n,0,s_1}(f_1)\right)+\E\left(T_{s_1}f_1(Z_{n}(0))\right)\]  
and for $k\geq 2$, by setting $s_0=0$
\begin{multline*}
\E\left(\prod_{i=1}^{k}f_i(Z_{n}(s_i))\right)=\E\left(\prod_{i=1}^{k-1}f_i(Z_{n}(s_i))I_{n,s_{k-1},s_k}(f_k)\right)\\+\E\left(\prod_{i=1}^{k-1}f_i(Z_{n}(s_i))T_{s_k-s_{k-1}}f_k(Z_{n}(s_{k-1}))\right).
\end{multline*}

Let us note that by assumption, if $h_1,h_2\in\mc{H}$  and $0<u<v$, the two functions  $T_{v-u}h_1$ and $h_1T_{v-u}h_2$  belong to $\mc{H}$ and $\E\left(T_{s_1}f_1(Z_{n}(0))\right)$ converges to $\E\left(f_1(Z(s_1))\right)$. Therefore, Proposition \ref{normanth1} follows by induction on $k$ if we show that $\E(|I_{n,u,v}(f)|)$ converges to $0$ as $n$ tends to $+\infty$ for every $0\leq u<v$ and $f\in\mc{H}$. 

Set  $h_n=\frac{v-u}{N_n}$, $u_j=u+jh_n$ and $g_j=T_{(N_n-j)h_n}f$ for $j\in\{0,\ldots, N_n\}$. With these notations, 
\begin{eqnarray*}I_{n,u,v}(f)&=&\E(g_{N_n}(Z_{n}(u_{N_n}))\mid \mathcal{F}_{n}(u_0))-g_0(Z_{n}(u_{0}))\\
&=&\sum_{j=0}^{N_n-1}\E\left[\E(g_{j+1}(Z_{n}(u_{j+1}))\mid \mathcal{F}_{n}(u_{j}))-g_{j}(Z_{n}(u_{j}))\mid \mathcal{F}_{n}(u_0)\right].
\end{eqnarray*}
Let us note that  $g_j=T_{h_n}g_{j+1}$. Therefore, 
$g_{j+1}(Z_{n}(u_{j+1}))-g_{j}(Z_{n}(u_{j}))$ can be decomposed as the difference of  
$g_{j+1}(Z_{n}(u_{j+1}))-g_{j+1}(Z_{n}(u_{j})))$ and 
$T_{h_n}g_{j+1}(Z_n(u_j))-g_{j+1}(Z_{n}(u_{j})))$. 

By applying Taylor expansion to $g_{j+1}(Z_{n}(u_{j+1}))-g_{j+1}(Z_{n}(u_{j}))$ and by Assumption \hyperlink{hypA1}{A1}, 
\begin{multline*}\E(g_{j+1}(Z_{n}(u_{j+1}))\mid \mc{F}_{n}(u_{j})) -g_{j+1}(Z_{n}(u_{j}))=\\
h_n\left(g'_{j+1}(Z_n(u_{j}))\alpha(Z_n(u_{j}))+\frac12 g''_{j+1}(Z_n(u_{j}))\beta(Z_n(u_{j}))\right)+R_{n,j}
\end{multline*}
with $|R_{n,j}|\leq \sum_{i=1}^{3}||g^{(i)}_{j+1}||_{\infty}e_{i,n}(u_{j},h_n,Z_{n}(u_{j}))$. 
By the properties of the semigroup $T_{t}$, we have a similar expansion of
 $T_{h_n}g_{j+1}(Z_{n}(u_{j}))-g_{j+1}(Z_{n}(u_{j}))$:
 \begin{multline*}T_{h_n}g_{j+1}(Z_{n}(u_{j}))-g_{j+1}(Z_{n}(u_{j}))=\\ h_n\left(g'_{j+1}(Z_n(u_{j}))\alpha(Z_n(u_{j}))+\frac12 g''_{j+1}(Z_n(u_{j}))\beta(Z_n(u_{j}))\right)+\bar{R}_{n,j}
 \end{multline*}
with $|\bar{R}_{n,j}|\leq \sum_{i=1}^{3}||g_{j+1}^{(i)}||_{\infty}|\epsilon_{i}(h_n,Z_{n}(u_{j}))|$. \\
By Assumption \hyperlink{hypA234}{A2},  $\sup_{j\in\{1,\ldots,N_n\}}\sum_{k=1}^{3}||g_{j}^{(k)}||_{\infty}$ is finite and depends only of $v-u$. By denoting it $M_{v-u}$, we have
\[\E(|I_{n,u,v}(f)|)\leq M_{v-u}\sum_{j=0}^{N_n-1}\sum_{i=1}^{3}\Big(\E\left(e_{i,n}(u_{j},h_n,Z_{n}(u_{j}))\right)+\E\left(\left|\epsilon_{i}(h_n,Z_{n}(u_{j}))\right|\right)\Big). \]
Finally,  Assumptions \hyperlink{hypA1}{A1} and \hyperlink{hypA234}{A3} imply that the upper bound converges towards 0 as $n$ tends to $+\infty$.

\newpage
\providecommand{\bysame}{\leavevmode\hbox to3em{\hrulefill}\thinspace}
\providecommand{\MR}[1]{%
  \href{http://www.ams.org/mathscinet-getitem?mr=#1}{MR#1}
}
\providecommand{\href}[2]{#2}

\end{document}